\documentclass[11pt]{article}
\usepackage{graphicx,psfrag}

\usepackage{fullpage}
\usepackage{amsmath,amsthm,amsfonts,amssymb,epsfig,verbatim, fge, bm}
\usepackage{xcolor}
\usepackage[colorlinks, citecolor=blue]{hyperref}
\usepackage[normalem]{ulem}
\usepackage{bbm}
\usepackage{circledsteps}
\usepackage{subcaption}
\usepackage{tikz}
\usetikzlibrary{positioning, calc, arrows.meta, shapes}
\newtheorem{thm}{Theorem}[section]
\newtheorem{lem}[thm]{Lemma}
\newtheorem{prop}[thm]{Proposition}
\newtheorem{cor}[thm]{Corollary}
\newtheorem{df}[thm]{Definition}

\newtheorem{ass}[thm]{Assumption}

\newtheorem{question}[thm]{Question}
\newtheorem{obs}[thm]{Observation}

\numberwithin{equation}{section}

\newcommand{\E}{\mathbf{E}}

\newcommand{\M}{\mathcal{M}}
\newcommand{\B}{\mathcal{B}}
\newcommand{\prob}{\mathbf{P}}

\newcommand{\R}{\mathbb{R}}

\newcommand{\W}{\mathsf{W}}

\DeclareMathOperator{\Var}{Var}

\newcommand{\Ber}{\mathrm{Ber}}
\renewcommand{\M}{\mathsf{M}}
\renewcommand{\B}{\mathsf{B}}

\renewcommand{\setminus}{\mathbin{\fgebackslash}}
\newcommand{\bB}{\mathbb{B}}

\newcommand{\cL}{\mathcal{L}}
\newcommand{\cC}{\mathcal{C}}
\newcommand{\dist}{\mathrm{dist}}

\begin{document}
	\title{Correlation Decay for Maximum Weight Matchings on Sparse Graphs}
	\author{Wai-Kit Lam\thanks{National Taiwan University. Email: waikitlam@ntu.edu.tw. The research of W.-K. L. is supported by the National Science and Technology Council in Taiwan grant 113-2115-M-002-009-MY3.} 
		\and Arnab Sen\thanks{University of Minnesota. Email: arnab@umn.edu.  The research of A.-S. is partly supported by Simons Foundation MP-TSM-00002716.} }
	\date{}
	\maketitle

\begin{abstract}
We study correlation decay for the maximum weight matching problem on sparse graphs with i.i.d.\ edge weights. We show exponential decay of correlations when the underlying graphs are locally tree-like with uniformly bounded degree and the edge weights are exponential. We also prove a polynomial rate of decay of correlations for any finite graph with maximum degree at most three, again for exponential edge weights. As consequences of the correlation decay property, we obtain the existence of the maximum weight matching on infinite graphs, local weak convergence of the maximum weight matching, and a law of large numbers for its total weight.
\end{abstract}

%\tableofcontents
\section{Introduction}
For a graph $G=(V(G),E(G))$, a matching $M\subseteq E(G)$ is a set of edges such that no two edges share a common vertex. When $G$ is equipped with edge weights $(w_e)_{e\in E(G)}$, the weight of a matching $M$, denoted by $w_M$, is defined by
\[
w_M := \sum_{e\in M} w_e .
\]
A maximum weight matching (MWM) of $G$, which we write as $\M_G$, is a matching of $G$ with maximum total weight, that is, 
$\M_G = \arg\max_{M} w_M$. 
Finding a MWM is a classical problem in combinatorial optimization and can be solved in polynomial time, for example, by Edmonds' blossom algorithm~\cite{Edmonds1965Paths}. To analyze random instances of this optimization problem, one may assume that the edge weights are i.i.d.\ from a common continuous distribution, so that the resulting MWM is almost surely unique. In the mean-field setting, the solution is expected to be replica symmetric, a feature that sets this combinatorial optimization problem apart from models such as random $k$-SAT and the Sherrington--Kirkpatrick spin glass, where replica symmetry breaking plays a central role. Based on the assumption of replica symmetry, M\'ezard and Parisi obtained remarkably detailed predictions for the minimal weight perfect matching in the mean-field setting, most notably for the random assignment problem on the complete bipartite graph~\cite{mezard1985replicas,mezard1986mean,mezard1987solution}. Several of these predictions were later established rigorously~\cite{aldous1992asymptotics,aldous2001zeta,wastlund2012replica}.

In this paper we focus on the sparse regime, where the underlying graph $G$ has uniformly bounded degree. The replica symmetric behavior of the MWM in this diluted setting remains largely unexplored. A central question is whether the MWM on sparse graphs exhibits decay of correlations, which would be a manifestation of replica symmetric behavior. One of the main aims of this work is to formulate this as a precise mathematical problem and to obtain partial results in this direction.

To formalize the correlation decay question, we introduce a few notations. For a graph $G=(V(G),E(G))$ and vertices $u,v\in V(G)$, we write $u\sim_G v$ if $(uv)\in E(G)$. For distinct edges $e,e'\in E(G)$, we write $e\sim_G e'$ if they share a common endpoint. When the underlying graph is clear from context, we omit the subscript and write $u\sim v$ and $e\sim e'$. For an edge $e = (uv)$ of a graph $G$ and $r\geq 0$, we define $\bB_e^r(G)$ to be the subgraph of $G$ induced by all the vertices in $V(G)$ whose graph distance from either $u$ or $v$ is at most $r$. Also, if $H$ is a subgraph of $G$, we define its (vertex) boundary as
\[
\partial H := \{v\in V(H): v\sim_G u \text{ for some $u\in V(G)\setminus V(H)$}\}.
\]

For $F \subseteq V(G)$, $G\setminus F$ denotes the induced subgraph of $G$ on the vertex set $V(G) \setminus F$. Similarly, for $E \subseteq E(G)$, we denote by $G \setminus E$ the subgraph $(V(G), E(G) \setminus E)$.

Throughout the paper, we work under the following assumption on the edge weights.
\begin{ass} \label{assumption: edge weight}
$(w_e)_{e\in E(G)}$ are i.i.d.\ nonnegative random variables from a continuous distribution $F$.
\end{ass}

Let $G$ be a locally finite, connected graph. For each $e \in E(G)$ and $r \ge 1$, define
\begin{equation}
    \label{eq: corr_decay}
    \varrho_r(G, e) := \E \sup_{A\subseteq \partial \bB_e^r(G)} |\mathbf{1}_{\{e \in \M_{\bB_e^r(G)}\}} - \mathbf{1}_{\{e \in \M_{\bB_e^r(G) \setminus A}\}}|.
    \end{equation}

Let $\mathcal{G}_{D}$ be the class of finite connected graphs with maximum degree bounded above by $D.$
\begin{question}\label{q:uniform_corr_decay} Fix any weight distribution $F$ satisfying Assumption~\ref{assumption: edge weight}. For any $D \ge 2$, is it true that
\begin{equation}\label{qes:corr_decay}
  \lim_{r \to \infty} \sup_{ G \in \mathcal{G}_{D}} \max_{ e \in E(G)}  \varrho_r(G, e) = 0\ ?  
\end{equation}
\end{question}
\noindent If the answer is yes,  one can also ask for quantitative bounds on the rate of convergence. 

For a finite graph $G$ and $e \in E(G),$ one can easily check (see Section~\ref{subsection:obs}) that
\begin{equation}
\label{eq: rough_corr_decay}
\E |\mathbf{1}_{\{e \in \M_G\}} - \mathbf{1}_{\{e \in \M_{\bB_e^r(G)}\}}| \le \varrho_r(G, e).
\end{equation}
So, if $\varrho_r(G, e) \approx 0$, then the event $\{e \in \M_G\}$ is approximately determined by only the weights of the local neighborhood of $e$. In other words, there is no long-range correlation in the MWM, so \eqref{qes:corr_decay} can be viewed as a form of correlation decay.

We believe that the answer to the above question is affirmative. First, the classical Heilmann--Lieb theory \cite{HeilmannLieb} shows that the (weighted) monomer-dimer model, which is a positive temperature relaxation of the weighted matching problem, has no phase transition at any temperature. In \cite{lamsenpositivetemp} it is shown, using van den Berg's disagreement percolation argument, that the disordered monomer-dimer model on bounded-degree graphs exhibits exponential decay of correlations. It is therefore reasonable to expect that replica-symmetric behavior continues to hold at zero temperature, or equivalently when $\beta = \infty$, for the disordered model. However, let us remark that the correlation decay exponent obtained in \cite{lamsenpositivetemp} tends to zero as $ \beta \to \infty$, so this result does not by itself yield correlation decay at zero temperature. Second,  there is no intrinsic algorithmic barrier to finding a MWM, since it can be computed in polynomial time. This is in contrast with the maximum-weight independent set problem, which is computationally hard, and the corresponding hard-core model is also known to have a phase transition.

\subsection{Main results}
Our main results give an affirmative answer to Question~\ref{q:uniform_corr_decay} for two subclasses of graphs with exponential edge weights: (a) locally tree-like graphs with uniformly bounded degree, and (b) graphs with maximum degree at most three. In both cases, we obtain explicit bounds on the rate of decay, stated in the next two theorems.
\begin{thm}
    \label{thm: temp_corr_decay}
    Let $G$ be a connected graph whose degree is bounded above by $D$. Suppose that the edge weights are i.i.d.\ exponential random variables with mean $1$. Let $e \in E(G)$ and $r\geq 3$ be such that $\bB_e^r(G)$ is a tree. 
    %Let $\partial \bB_e^R(G)$ denote the inner vertex  boundary of $\bB_e^R(G)$.
    Then we have
    \[
    \varrho_r(G, e)  \leq 2D(1 - (2D)^{-D})^{r-2}.
    \]
\end{thm}
In \cite{Gamarnik}, correlation decay is established for random $d$-regular graphs and Erd\H{o}s--R\'enyi graphs $G(n,\lambda/n)$ with exponential weights. The local limits of these graphs are infinite $d$-regular trees or Poisson Galton--Watson trees, respectively, which are self-similar or stochastically self-similar. On such trees, the cavity computation yields a recursive distributional equation, and correlation decay follows once one proves that this equation has a unique solution, as shown in \cite{Gamarnik} for exponential weights. Recently, the MWM problem was studied in \cite{enriquez2025optimal} for finite graphs whose local limit is a unimodular Galton--Watson tree. While \cite{enriquez2025optimal} does not explicitly address correlation decay, it shows that the recursive distributional equation has a unique solution for any non-atomic measure with a nontrivial absolutely continuous part. Let us emphasize that these results do not typically extend to deterministic irregular trees, where each vertex yields a different distributional equation. As far as we know, correlation decay arguments on irregular trees are not common in the literature. 
Let us also mention that in the positive temperature case, for the (unweighted) monomer-dimer model on finite graphs whose local limit is a unimodular Galton-Watson tree, the limiting monomer density was computed via cavity recursion in \cite{alberici2014solution}.

\begin{thm}
    \label{thm: corr_decay_degree_3}
    Let $G$ be a connected graph with maximum degree at most $3$. Suppose that the edge weights are i.i.d.\ exponential random variables with mean $1$. Let $e\in E(G)$. Then  for any $r\geq 1$ and for any $\varepsilon  \in (0, 1)$,
    \[
   \varrho_r(G, e)   \leq 6 \big((1-\varepsilon)^{r} + \varepsilon \big).
    \]
    In particular, setting $\varepsilon = \log{r}/r$, we obtain that 
    \[
    \varrho_r(G, e)  \leq  \frac{18\log{r}}{r} \quad \text{for all $r\geq 2$.}
    \]
\end{thm}
While the assumption that the maximum degree is bounded above by $3$ is very restrictive, the above theorem does include the interesting example of the two-dimensional hexagonal lattice (which contains short cycles).

A result of similar flavor was proved in \cite{gamarnik2010ptas} for the maximum weight independent set. It constructed a randomized PTAS (polynomial-time approximation scheme) that approximates the total weight of an optimal independent set on arbitrary graphs of maximum degree $3$ with exponential weights.
Their proof established correlation decay not on the original graph, but on a random subgraph obtained via a high-density site percolation, which still sufficed to approximate the total weight within a $(1\pm \epsilon)$ factor. By passing to this subgraph, they reduced the average degree to at most $3-\delta$ for some $\delta>0$, which played a crucial role in proving contractivity of the cavity recursion. In our case, we carried out a more delicate analysis of the recursion that allowed vertices of degree $3$.  Interestingly, they also showed that, for sufficiently large maximum degree, no PTAS for approximating the total weight exists unless P$=$NP, in stark contrast with maximum weight matching, where the total weight can be computed exactly in polynomial time.

A recent remarkable paper \cite{KrishnanRay} proves a two-point correlation decay for the maximum weight matching on the $d$-dimensional torus $\mathbb{T}_n^d$ for a class of absolutely continuous weight distributions, including the Gaussian law. However, their proof relies heavily on unimodular techniques, so no explicit rate of correlation decay is obtained, and the argument does not seem to generalize to non-unimodular graphs.

The proofs of Theorems~\ref{thm: temp_corr_decay} and \ref{thm: corr_decay_degree_3} rely heavily on the memoryless property of the exponential distribution, and extending these arguments beyond the exponential case remains an open problem. It would also be interesting to investigate whether the bounded degree assumption in Theorem~\ref{thm: temp_corr_decay} can be replaced by a `stochastically bounded degree' assumption, which would cover more general (random) sparse graphs such as $G(n,\lambda/n)$.

Instead of `uniform' correlation decay for finite graphs as in Question \ref{q:uniform_corr_decay}, one may also ask whether the correlation decay holds for an infinite graph, which is stated in the assumption below.

\begin{ass}[Correlation decay assumption for infinite graph] \label{assumption: corr_decay_general_graph}
A  locally finite and connected graph  $G$ is said to satisfy the correlation decay if  for each $e \in E(G)$, we have
    \begin{equation}
    \label{eq: corr_decay_general_graph}
    \varrho_r(G, e) \to 0 \quad \text{as $r\to\infty$.}
    \end{equation}
\end{ass}
\noindent Note that the correlation decay assumption trivially holds when $G$ is a finite graph.  By Theorem~\ref{thm: temp_corr_decay} and Theorem~\ref{thm: corr_decay_degree_3}, if the weights are exponential, then the correlation decay assumption holds 
when $G$ is an infinite tree of bounded degree or $G$ is an infinite graph with maximum degree $3$.

Next, we derive several consequences of the correlation-decay assumption. First, under correlation decay, the notion of the MWM extends naturally to infinite graphs.

\begin{prop}[MWM on an infinite graph]
\label{thm: MWM_inf_graph}
    Let $G$ be a locally finite infinite connected graph satisfying Assumption~\ref{assumption: corr_decay_general_graph}. Fix $e \in E(G)$. Let $(G_n)$ be a sequence of finite graphs that increases to~$G$. The following statements hold almost surely.
 	\begin{enumerate}
          \item[(a)]   \[
    \lim_{n\to\infty} \mathbf{1}_{\{e\in \M_{G_n}\}} \quad \text{exists.}
    \]
    \end{enumerate}
Let $\M_G \subseteq E(G)$ be the set of all edges $e\in E(G)$ such that $\lim_{n \to \infty} \mathbf{1}_{\{e\in \M_{G_n}\}} = 1$.
    \begin{enumerate}
	\item[(b)]  $\M_G$ does not depend on the choice of $(G_n)$. More precisely, if $(G_n')$ is another sequence of finite graphs that increases to $G$, then 
		\[
		\lim_{n \to \infty} \mathbf{1}_{\{e \in \M_{G_n'}\}} = \lim_{n \to \infty} \mathbf{1}_{\{e \in \M_{G_n}\}}. 
		\]
  
		\item[(c)] $\M_G$ is a matching.
		\item[(d)] (local optimality) For any matching $M$ of $G$ such that 
        $|M \triangle \M_G| < \infty$ and $M\neq \M_G$, we have 
        \[\sum_{ e \in \M_G \setminus M}  w_e > \sum_{ e \in M \setminus \M_G} w_e.\]
        	\end{enumerate}
\end{prop}
We refer to the matching $\M_G$ constructed above as the MWM of the infinite graph $G$, which can be realized as a measurable function of the edge weights. By (b), $\M_G$ can be viewed as the unique infinite-volume ground state.

Let us mention two related results from the literature. In \cite{KrishnanRay}, the authors establish the uniqueness of the `metastate' for MWM on $\mathbb{Z}^d$. In their framework, instead of proving that for almost every quenched realization of the weights $w$ there is a single optimal matching satisfying local optimality, they construct a joint distribution on $w$ and an optimal matching such that, for almost every $w$, the conditional distribution of the optimal matching given $w$ is supported on a single optimal matching. In \cite{enriquez2025optimal}, the authors construct a unimodular matching on a unimodular Galton--Watson tree. This matching is unique in law and is optimal among all unimodular matchings with respect to the weight density at the root. However, their construction is again not quenched, and only holds in distribution.

Our next result shows that if a sequence of finite (possibly random) graphs converges locally weakly and the limiting unimodular random graph $G$ satisfies the correlation decay assumption for almost every realization of $G$, then the MWM on the finite graphs converges locally weakly to the MWM on the limit. We will refer the reader to Appendix~\ref{sec: local_weak_convergence} for the definition of local weak convergence and unimodularity.

\begin{prop}[Local weak convergence of MWM]
    \label{lem: local_weak_convergence_matching}
    Let  $(G_n)$ be a sequence of random graphs such that $(G_n, o_n)$  converges locally weakly to $(G, o)$, where $o_n$ is uniformly chosen from $V(G_n)$.
    Suppose that almost surely under the law of the random graph $(G, o),$ $G$ satisfies the correlation decay assumption. Then the sequence $(G_n, o_n, (w_e)_{e\in E(G_n)}, (\mathbf{1}_{\{e \in \M_{G_n}\}})_{e \in E(G_n)})$ converges locally weakly to $(G, o, (w_e)_{e\in E(G)}, (\mathbf{1}_{\{e \in \M_{G}\}})_{e \in E(G)})$.
\end{prop}
This yields the following law of large numbers for the MWM’s total weight and edge count. Let $\W_G := w_{\M_G}$ denote the weight of the MWM of $G$, i.e., 
\[
\W_G = \sum_{e \in E(G)} w_e\mathbf{1}_{\{e \in \M_G\}}.
\]
\begin{cor}
	\label{thm: lln}
Let $(G_n)$ be a sequence of random graphs and  let $(G_n,o_n)$ converge locally weakly to $(G,o)$ where $o_n$ is  chosen uniformly from $V(G_n)$. Assume  that 
$\sup_n \E \tfrac{|E(G_n)|}{| V(G_n)|} < \infty$ and $\E|w_e|^p<\infty$ for some $p>1$. Further assume that, almost surely under the law of $(G,o)$, the limit graph $G$ satisfies the correlation decay assumption. Then
	\[
	\lim_{n\to\infty} \E \frac{\W_{G_n}}{|V(G_n)|} = \frac{1}{2}\E \left[ \sum_{v: v\sim o}w_{(ov)} \mathbf{1}_{\{(ov) \in \M_G\}} \right],
	\]
    and 
    \[
	\lim_{n\to\infty} \E \frac{|\M_{G_n}|}{|V(G_n)|} = \frac{1}{2}\E \left[ \sum_{v: v\sim o}\mathbf{1}_{\{(ov) \in \M_G\}} \right],
	\]
    where $|\M_{G_n}|$ is the number of edges in the matching $\M_{G_n}$ and the expectation $\E$ is taken over both the randomness in the graph and the edge weights.
\end{cor}

As an immediate consequence of Corollary~\ref{thm: lln} and Theorem~\ref{thm: temp_corr_decay}, the law of large numbers above applies to random graphs with i.i.d.\ degrees. More precisely, let $G_n$ be a uniformly chosen graph on $n$ vertices with prescribed degree sequence $(d_1,\ldots,d_n)$, where the $d_i$ are i.i.d.\ with distribution $\pi$  with bounded support. It is well known that $(G_n,o_n)$ converges locally weakly to a rooted unimodular Galton--Watson tree (see, e.g., \cite[Theorem~4.6]{vdHofstad}). Hence, if we assign i.i.d.\ $\mathrm{exp}(1)$ weights to the edges of $G_n$, the conclusion of Corollary~\ref{thm: lln} holds. 

On the other hand,  let $(G,o)$ be the hexagonal lattice, rooted at an arbitrary vertex $o$, and set $G_n=\bB_o^n(G)$, the subgraph of $G$ induced by all the vertices in $V(G)$ whose graph distance from $o$ is at most $n$. Then $(G_n,o_n)$ converges locally weakly to $(G,o)$. If we assign i.i.d.\ $\mathrm{exp}(1)$ weights to $E(G_n)$, Theorem~\ref{thm: corr_decay_degree_3} together with Corollary~\ref{thm: lln} yields a law of large numbers for the MWM on the hexagonal lattice.

The local weak limit of the MWM, as stated in Proposition~\ref{lem: local_weak_convergence_matching}, was established in \cite{enriquez2025optimal} without any correlation-decay assumption for general atomless weight distributions when the limit $G$ is a unimodular Galton--Watson tree. As a consequence, the same law of large numbers for random graphs with i.i.d.\ degrees is also proved in \cite{enriquez2025optimal} for general atomless weight distributions. A law of large numbers for the MWM had previously been established in \cite{Gamarnik} for random $d$-regular graphs and Erd\H{o}s--R\'enyi graphs $G(n,\lambda/n)$ with exponential weights.

Finally, Theorems~\ref{thm: temp_corr_decay} and \ref{thm: corr_decay_degree_3} allow us to obtain a central limit theorem for $ \W_{G_n} $ for any sequence of nonrandom graphs $(G_n)$ with  $|V(G_n)| \to \infty$ and with i.i.d.\ exponential edge weights, provided either (a) $G_n$ has uniformly bounded degree and girth $ g_n \to \infty $, or (b) $G_n$ has maximum degree at most $3$. In this setting, we have
\[ \frac{\W_{G_n} - \E \W_{G_n}}{\sqrt{\mathrm{Var}(\W_{G_n})}} \stackrel{d}{\to} N(0, 1). \]
The main input for this result is correlation decay. Once correlation decay is available, the proof follows the general approach of \cite{lamsenpositivetemp}, which applies Chatterjee's perturbative version of Stein's method \cite{Chatterjee14} and yields an explicit error bound in Kolmogorov--Smirnov distance. We omit the details. Central limit theorems for MWMs were also proved by a similar method in \cite{Cao} for Erd\H{o}s–R\'enyi graphs $G(n,\lambda/n)$ and in \cite{KrishnanRay} for $\mathbb{T}_n^d$.

\subsection{A brief sketch of proof}
We now give a brief outline of the proof of Theorem~\ref{thm: temp_corr_decay}. For a vertex $v$ in a graph $G$, we define its bonus by
\[
\B_v(G) := \W_G - \W_{G \setminus \{v\}} \in [0,\infty).
\]
Let $T$ be a finite tree, and let $v$ be a vertex of $T$. Denote by $T_v$ the subtree of $T$ rooted at $v$. Let $v_1,\ldots,v_d$ be the children of $v$, and let $T_1,\ldots,T_d$ be the corresponding subtrees rooted at $v_1,\ldots,v_d$. By the cavity recursion, the bonus at the root $v$ can be expressed in terms of the bonuses of these subtrees as
\[
\B_v(T_v) = \max_{1 \le i \le d} \bigl(w_{(v v_i)} - \B_{v_i}(T_i)\bigr)_+,
\]
where the bonuses $\B_{v_1}(T_1),\ldots,\B_{v_d}(T_d)$ are independent, since we are working on a tree. This distributional recursion depends on $d$, the number of children of $v$, so for an irregular tree, we do not have a single fixed point equation (in the limit), but rather a family of equations that vary from vertex to vertex. Correlation decay follows if we can show that the distribution of the bonus at the root of a large depth tree is only weakly sensitive to the boundary condition at the leaves. While this approach is quite natural, it is very challenging to analyze these recursions on the space of probability measures. 
  
Instead, under the assumption of exponential edge weights, we use the memoryless property to derive a new recursion for the quantities
\[
p(e) := \prob\bigl(w_e > \B_v\bigr),
\]
defined for an edge $e = (u v)$ where $u$ is the parent of $v$, in terms of $p(e_1),\ldots,p(e_d)$ for the child edges $e_i = (v v_i)$:
\begin{align*}
p(e)
&= \E\left[\frac{1}{1 + \xi(e_1) + \xi(e_2) + \cdots + \xi(e_d)}\right]
=: \Phi\bigl(p(e_1), p(e_2), \ldots, p(e_d)\bigr),
\end{align*}
where $\xi(e_1),\ldots,\xi(e_d)$ are independent with $\xi(e_i) \sim \mathrm{Ber}\bigl(p(e_i)\bigr)$. 

This reduces the recursion from the space of probability measures to a recursion on reals, which simplifies the analysis considerably. One way to obtain correlation decay is now to show that the map obtained by composing the operators $\Phi$ along the tree, which takes as input the boundary values at the leaves and outputs the $p$-value for the root edge, is a contraction. However, this naive approach fails, since
\[
\sup_{p \in [0,1]^d} \bigl\|\nabla \Phi(p_1,\ldots,p_d)\bigr\|_1
\]
grows with $d$ and cannot be bounded by $1$ for large (fixed) $d$. To overcome this, we perform the change of variables $q = -\log p$, which leads to a new recursion
\[
q(e) = \Psi\bigl(q(e_1), q(e_2), \ldots, q(e_d)\bigr)
:= - \log \Phi\bigl(e^{-q(e_1)}, \ldots, e^{-q(e_d)}\bigr).
\]
After this transformation, the resulting map, obtained by composing the operators $\Psi$, becomes contractive, which yields correlation decay. 

\subsection{Organization of the paper}
The rest of the paper is organized as follows. In Section~\ref{sec:prelim}, we record some elementary but important results on MWM with deterministic edge weights. We introduce the definition of bonus and explain how the MWM can be characterized in terms of bonuses. We derive the cavity recursion satisfied by the bonuses and prove some local bounds on the bonus. In Section~\ref{sec:cor_decay}, we prove our main correlation decay results, Theorems~\ref{thm: temp_corr_decay} and \ref{thm: corr_decay_degree_3}. In Section~\ref{sec_mwm_infinite}, we prove the results that follow from correlation decay, namely Propositions~\ref{thm: MWM_inf_graph} and \ref{lem: local_weak_convergence_matching} and Corollary~\ref{thm: lln}. The appendix collects some basic definitions and results on local weak convergence and unimodularity.\\

\noindent\textbf{Acknowledgment.} AS thanks David Gamarnik for helpful discussions, and Kesav Krishnan for pointing out the reference \cite{{alberici2014solution}}.

\begin{comment}
    \textcolor{blue}{Do we need it?} \paragraph{Notation.}
Our underlying graph could be random, and $\prob_G$ denotes the probability measure for the randomness of the graph (and $\E_G$ will be the corresponding expectation). $\prob_w$ (respectively $\E_w$) is the probability measure (respectively expectation) with respect to the edge weights. $\prob$ (respectively $\E$) will be the probability measure (respectively expectation) with respect to both randomness. Sometimes we will simply write $\prob$ for the probability measure for the weights, when the underlying graph $G$ is not random.
\end{comment}

\section{Preliminaries}\label{sec:prelim}

In this section, we collect some useful facts about the MWM on a finite graph $G$. The results presented here do not rely on any randomness in the edge weights or the graph; however, we assume the weights are ``generic'', ensuring that $G$ has a unique MWM.

\subsection{Restriction of the maximum weight matching}\label{subsection:obs}

\begin{obs}
\label{obs: forbidden_vertices}
Let $G$ be a finite graph, and fix $e \in E(G)$ and $r\ge 1$. Let $\M_G$ be the maximum-weight matching of $G$, and let $\mathsf{A}$ be the set of vertices $v$ for which there exists $u$ with $(uv)\in \M_G$ but $(uv)\notin E(\bB_e^r(G))$. Then
\begin{equation}\label{eq: forbidden_vertices}
\mathbf{1}_{\{e\in \M_G\}}=\mathbf{1}_{\{e\in \M_{\bB_e^r(G)\setminus \mathsf{A}}\}}.
\end{equation}
\end{obs}
\begin{proof}
Write $B=\bB_e^r(G)$. Then the union
$\M_{B\setminus \mathsf{A}} \cup (\M_G\setminus E(B))$
is a matching of $G$. Since $\M_G = (\M_G \cap E(B) ) \cup  (\M_G \setminus E(B))$ is the MWM of $G$,
\[
w_{\M_G\cap E(B)} \ \ge\ w_{\M_{B\setminus \mathsf{A}}}.
\]
Conversely, $\M_G\cap E(B)$ is a matching on $B\setminus \mathsf{A}$ on which $\M_{B\setminus \mathsf{A}}$ is the optimal matching, so we have the reverse inequality as well.  Hence the weights are equal, and since the weights are generic, the matchings coincide: $
\M_G\cap E(B)=\M_{B\setminus \mathsf{A}}.$ This yields \eqref{eq: forbidden_vertices}.
\end{proof}

\subsection{Augmenting paths}
The MWM on a finite graph can be characterized using augmenting paths, defined as follows.
\begin{df}
\label{df:augmenting_path}
Let $G$ be a finite graph with fixed (non-random) edge weights $(w_e)_{e \in E(G)}$, and let $M$ be a matching.  
An augmenting path for $M$ is a finite self-avoiding path or cycle  
\[
P = e_1 \sim e_2 \sim \cdots \sim e_n
\]
where $e \sim e'$ means that the edges are adjacent in $G$ (and, if $P$ is a cycle, we also have $e_n \sim e_1$), such that the following hold:
\begin{enumerate}
    \item The edges of $P$ alternate between those in $M$ and those not in $M$.  In particular, if $P$ is a cycle, then $n$ must be even.
    \item Writing $e_i = (x_{i-1}, x_i)$ for $i = 1, \ldots, n$, if $P$ is not a cycle and if $e_1$ (respectively $e_n$) is not in $M$, then $x_0$ (respectively $x_n$) is unmatched in $M$.
    \item The edge weights on $P$ satisfy
    \[
        \sum_{e \in P \cap M} w_e \;<\; \sum_{e \in P \cap M^c} w_e.
    \]
\end{enumerate}
\end{df}
It turns out that a matching has maximum weight if and only if there is no augmenting path for the matching.
\begin{lem}
	\label{lem: augmenting}
	Let $G$ be a finite weighted graph and let $M$ be a matching of $G$. Then $M$ is a MWM if and only if there is no augmenting path for $M$.
\end{lem}
\begin{proof}
Suppose that there is an augmenting path $P$ for $M$.  
Then 
\[
(M \setminus (M \cap P)) \cup (M^c \cap P)
\]
is a matching whose weight is strictly larger than that of $M$, so $M$ cannot be a MWM.

Conversely, suppose that $M$ is not a MWM.  
Then there exists a matching $M'$ of $G$ such that $w_{M'} > w_M$.  
Note that $M \triangle M'$ is a disjoint union of paths and cycles.  
Clearly, at least one such path or cycle $P$ must satisfy
\[
\sum_{e \in P \cap M} w_e \;<\; \sum_{e \in P \cap M^c} w_e.
\]
We claim that such a $P$ is an augmenting path for $M$.  
It is immediate that the edges of $P$ alternate between those in $M$ and those not in $M$.

\smallskip
\noindent\textbf{Case 1:} $P$ is a cycle.  
If $n$ were odd, then both $e_1$ and $e_n$ would belong to the same matching (either $M$ or $M'$), which is impossible.  
Hence $n$ must be even.

\smallskip
\noindent\textbf{Case 2:} $P$ is not a cycle.  
Without loss of generality, we only need to check that if $e_1 = (x_0, x_1) \in M'$, then $x_0$ is unmatched in $M$.  
Suppose, to the contrary, that $x_0$ is matched in $M$.  
Then there exists an edge $e_0 = (x_{-1}, x_0) \in M$.  
Note that $e_0$ must also be in $M'$, since $e_0 \notin M \triangle M'$.  
This is impossible, because $e_0 \sim e_1$ and both $e_0, e_1 \in M'$.

\smallskip
Therefore, $P$ satisfies all the properties of an augmenting path.  
In other words, if $M$ is not a MWM, then there exists an augmenting path for $M$.
\end{proof}

\subsection{Cavity recursion of MWM}
In the study of MWM, a standard and widely used tool is the concept of the \emph{bonus} (or \emph{cavity}).  
This notion, introduced in \cite{AldousSteele} to study the MWM on random trees, is defined as follows.
\begin{df}
	For any finite graph $G$ and $v\in V(G)$, we define the bonus at the vertex $v$ in $G$ by $\B(v,G) = \W_G - \W_{G\setminus\{v\}} \in [0, \infty)$.
\end{df}
In words, the bonus of a vertex represents the loss in the weight of the optimal matching when no edge incident to $v$
 is allowed to be used. Thus, the bonus is always nonnegative. Also, a vertex $v\in V$ is covered by one of the edges of  $\M_G$ if and only if $\B(v,G)>0$. 
 
 It is easy to check that the weight of the MWM satisfies the following recursion:
for any finite subgraph $H$ of a graph $G$ and for any vertex $v \in V(H)$, 
\begin{equation*} \label{eq:key_recursion}
 \W_H = \max\big ( \W_{H\setminus\{v\}}, \max_{u: (uv) \in E(H) } \big(w_{(uv)} + \W_{H \setminus \{u, v\}}\big)\big).
\end{equation*}
Subtracting $\W_{H\setminus\{v\}}$ from both sides of the above identity yields the recursion: 
\begin{equation}
\label{eq: bonus_recursion}
\B(v, H) = \max_{u: (uv) \in E(H)} \big(w_{(uv)} - \B(u, H\setminus\{v\})\big)_+.
\end{equation}
where  $(\;\cdot\;)_+ := \max(\cdot, 0)$. Although the relation \eqref{eq: bonus_recursion} provides a recursive distributional equation for bonuses on any finite graph, its rigorous analysis is challenging because the bonuses on the right-hand side of \eqref{eq: bonus_recursion} are not independent in general. 
This issue disappears when the underlying graph is a tree.  
Let $T$ be a rooted tree. For $v \in V(T)$, denote by $T_v$ the subtree rooted at $v$, and let $v_1, v_2, \ldots, v_{d_v}$ be the children of $v$ in $T$.  
Then \eqref{eq: bonus_recursion} can be rewritten as  
\begin{equation}
\label{eq: bonus_recursion_tree}
\B(v) = \max_{1 \le i \le d_v} \big(w_{(v v_i)} - \B(v_i)\big)_+,
\end{equation}
where $\B(v)$ denotes $\B(v, T_v)$.  
The advantage of the recursion \eqref{eq: bonus_recursion_tree} is that the bonuses and the weights appearing on the right-hand side are all independent.

\subsection{Characterizing edges in  MWM}
\label{sec: characterization}

We present a characterization of $\{e \in \M_G\}$ for a finite graph $G$, inspired by the message-passing algorithm of \cite{enriquez2025optimal}. While their focus is on finite trees, a similar, though slightly less symmetric, characterization can be derived for general graphs.

Write $e = (uv)$. If $e \in \M_G$, then 
\[
\W_G = w_{(uv)} + \W_{G\setminus \{u, v\}}.
\]
On the other hand, if $e\not\in \M_G$, then
\[
\W_G = \W_{G\setminus \{(uv)\}}.
\]
Thus, $e\in \M_G$ if and only if
\[
w_{(uv)} + \W_{G\setminus \{u, v\}} - \W_{G\setminus \{(uv)\}} > 0,
\]
or, equivalently,
\[
w_{(uv)} > \W_{G\setminus \{(uv)\}} - \W_{G\setminus \{u, v\}}.
\]
Note that the right side of the above inequality is nonnegative. Furthermore,
\begin{align*}
    \W_{G\setminus \{(uv)\}} - \W_{G\setminus \{u, v\}} &= \W_{G\setminus \{(uv)\}} - \W_{G\setminus\{u\}} + \W_{G\setminus \{u\}} - \W_{G\setminus \{u, v\}}\\
    &= \B(u, G \setminus \{(uv)\}) + \B(v, G\setminus\{u\}),
\end{align*}
where in the last equality we used the fact that $G\setminus\{u\} = (G\setminus \{(uv)\}) \setminus \{u\}$. Thus, 
\begin{equation}
\label{eq: equivalence}
    e\in \M_G \iff w_{(uv)} > \B(u, G \setminus \{(uv)\}) + \B(v, G\setminus\{u\}).
\end{equation}
We remark that, on a general finite graph, the two bonus terms on the right-hand side may be dependent, but both are independent of $w_{(uv)}$.

\begin{figure}[ht]
\centering
\begin{subfigure}{.5\textwidth}
  \centering
  \includegraphics[width=0.8\linewidth]{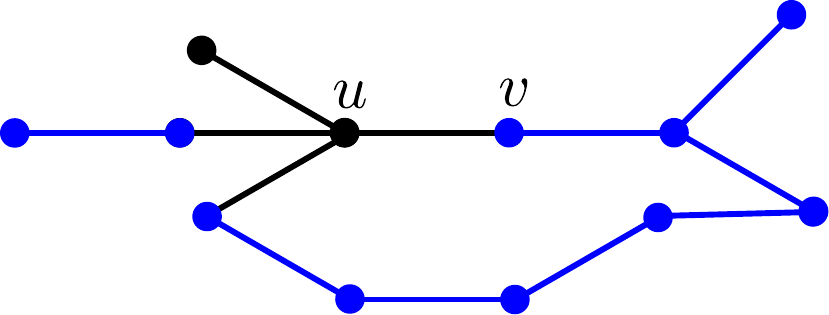}
  \caption{General graph $G$}
\end{subfigure}%
\begin{subfigure}{.5\textwidth}
  \centering
  \includegraphics[width=0.8\linewidth]{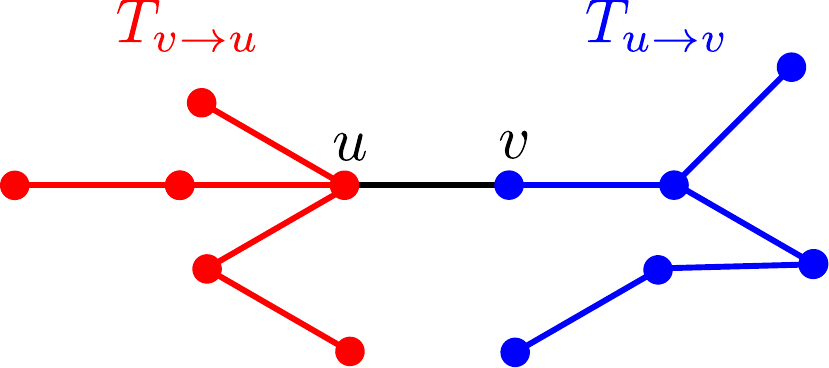}
  \caption{The case when $G$ is a tree}
\end{subfigure}
\caption{In (a), the blue subgraph is $G\setminus \{u\}$. In this case, its connected component containing $v$ overlaps with $G\setminus \{(uv)\}$, so the bonuses $\B(u, G \setminus \{(uv)\})$ and $\B(v, G\setminus\{u\})$ are not independent of each other, but they are independent of the edge weight $w_{(uv)}$. In (b), the graph $G$ is a tree, and in this case $\B(u, G \setminus \{(uv)\})$ is exactly $\B(u, T_{v\to u})$; similarly, $\B(v, G\setminus\{u\}) = \B(v, T_{u\to v})$.}
\end{figure}

When the graph $G$ is a tree, we can express \eqref{eq: equivalence} in a more symmetric form. Let $T_{v\to u}$ and $T_{u\to v}$ denote the subtrees rooted at $u$ and $v$, respectively, after removing the edge $(uv)$ from $G$. Then \eqref{eq: equivalence} becomes
\begin{equation}
    \label{eq: equivalence_tree}
    e\in \M_G \iff w_{(uv)} > \B(u, T_{v\to u}) + \B(v, T_{u\to v}).
\end{equation}
Moreover, since the edges of $T_{v\to u}, T_{u\to v}$ are disjoint from each other and also distinct from $(uv)$. Therefore, $ w_{(uv)}, \B(u, T_{v\to u}), \B(v, T_{u\to v}) $ are all independent.

\subsection{Local bounds for bonuses}  
\label{sec: local_bounds}
Fix a finite subgraph $H$ of $G$ with $|V(H)| \ge 1$. For any vertex $u \in V(H)$, let $d_u(H)$ denote the degree of $u$ in $H$, and  list its neighbors as $u_1, u_2,\ldots, u_{d_u(H)}.$

Given $r \ge 0$ and $u \in V(H)$, we now inductively define two local quantities, $\B^0_r(u, H)$ and $\B^\infty_r(u, H)$, which depend only on the edge weights within the ball of radius $r$ around $u$ in $H$, and serve as lower and upper bounds for $\B(u, H)$ depending on the parity of $r$. First assume that $d_u(H) \ge 1.$ Let
\begin{align*}
    \B_0^0(u, H) &= 0, \\
    \B_r^0(u, H) &= \max_{ 1 \le i \le d_u(H)}  (w_{(u u_i)}- \B_{r-1}^0(u_i, H \setminus \{u\}  )_+, \ \text{ for } r \ge 1.
\end{align*}
Similarly, let
\begin{align*}
    \B_0^\infty(u, H) &= \infty, \\
    \B_r^\infty(u, H) &= \max_{ 1 \le i \le d_u(H)}  (w_{(u u_i)}- \B_{r-1}^\infty(u_i, H \setminus \{u\}  )_+, \ \text{ for } r \ge 1.
\end{align*}
We set $\B_r^0(u, H)$ and $\B_r^\infty(u, H)$ to be $0$ if $d_u(H) = 0$.
Owing to the anti-monotonicity of the bonus recursion \eqref{eq: bonus_recursion}, we can easily show by induction that
\begin{align*}
    \B_{r}^0(u, H) \leq \B(u, H) \leq \B_{r}^\infty(u, H) & \quad \text{if $r$ is even,}\\
    \B_{r}^0(u, H) \geq \B(u, H) \geq \B_{r}^\infty(u, H) & \quad \text{if $r$ is odd.}
\end{align*}
  Furthermore, another simple induction shows that 
  for any $u \in V(H),$ $r \ge 0$,  and  $A\subseteq \partial \bB_u^r(H)$,
  \begin{align*}
      \B_{r}^0(u, H) &\le \B_{r}^0(u, H \setminus A) \ \ \text{ if $r$ is even,} \\
      \B_{r}^0(u, H) &\ge \B_{r}^0(u, H \setminus A) \ \ \text{ if $r$ is odd,}
  \end{align*}  
  and the inequality reverses for $\B_r^\infty.$ Combining, we obtain  
\begin{align}
    \label{eq: antimonotone_stronger}
    \begin{split}
    \B_{r}^0(u, H) \le \B_{r}^0(u, H \setminus A) \leq \B(u, H \setminus A) \le  \B_{r}^\infty(u, H \setminus A)  \leq \B_{r}^\infty(u, H) & \quad \text{if $r$ is even,}\\
    \B_{r}^0(u, H) \ge   \B_{r}^0(u, H \setminus A) \geq \B(u, H \setminus A) \geq \B_{r}^\infty(u, H \setminus A) \ge  \B_{r}^\infty(u, H) & \quad \text{if $r$ is odd.}
    \end{split}
\end{align}

\section{Correlation decay}\label{sec:cor_decay}
\subsection{Correlation decay on a tree}
Let $T$ be a rooted tree with root $o^\dagger$, where the degree of $o^\dagger$ is one, and let $o$ be the unique neighbor of $o^\dagger$.  
For any vertex $v\neq o^\dagger$, let $v^\dagger$ denote its parent, i.e., the unique neighbor of $v$ that is closer to the root $o^\dagger$.  
For a vertex $v$, denote its set of children by $\cC_v=\{v_1,\dots,v_{d_v}\}$, which may be empty.

For any $v\neq o^\dagger$, let $T_v$ be the subtree of $T$ containing $v$ obtained by deleting its parent $v^\dagger$ from $T$.  
We adopt the convention that $T_{o^\dagger}=T$. To each vertex $v\neq o^\dagger$ of $T$, associate the edge $e_v=(v\,v^\dagger)$, and refer to $e_{v_1},\dots,e_{v_{d_v}}$ as the \emph{child edges} of $e_v$.

Let $\cL_T$ be the set of leaves of $T$, i.e., the non-root vertices of degree one.  
Define the height of $T$ as
\[
r:=\max_{v\in\cL_T}\dist(v,o^\dagger),
\]
where $\dist$ denote the graph distance in $T$, and let
\[
\cL_T^{\,r}:=\{v\in\cL_T:\dist(v,o^\dagger)=r\}
\]
be the subset of leaves at distance exactly $r$ from the root. We assume that $r \ge 4.$

We define the recursion of the bonuses on $T$ with a boundary condition 
$(a_v)_{v \in \mathcal{L}^r_T} \in [0, \infty)^{|\mathcal{L}^r_T|}$
as follows:
 \begin{align}\label{eq:tree_rec}
 \B^a(v) &= \max_{1 \le i \le d_v} \big(w_{e_{v_i}} - \B^a(v_i)\big)_+, \ \ \ \text{for } v \not \in \mathcal{L}_T,  \\
 \B^a(v) &= a_v \in [0, \infty) \ \ \ \text{for } v \in \mathcal{L}_T^r, \ \ \text{ and } \ \  \B^a(v) = 0  \ \ \ \text{for } v \in \mathcal{L}_T \setminus \mathcal{L}_T^r. \nonumber
 \end{align}
It is straightforward to check that, for any given boundary condition, the recursion \eqref{eq:tree_rec} admits a unique solution $(\B^a(v))_{v \in V(T)}$. In the case of the zero boundary condition, i.e., when $a \equiv 0$, the solution $(\B^0(v))_{v \in V(T)}$ coincides with the true bonuses, namely,
\[
\B^0(v) = \B(v) := \B\big(v, T_v\big), \quad \forall\, v \in V(T),
\]
since, for each non-leaf vertex, the recursion \eqref{eq:tree_rec} is exactly the same as \eqref{eq: bonus_recursion_tree}. 

 Let us make a simple yet important observation, which is a consequence of the recursion \eqref{eq:tree_rec} being anti-monotone. For any $ v \in V(T)$, 
\begin{equation}\label{anti_monotone}
(-1)^{r - \mathrm{dist}(o^\dagger, v)} \B^a(v)  \text{ is an increasing function of } (a_u)_{u \in \mathcal{L}^r_T}.
\end{equation}

For every edge $e_v = (v^\dagger v)$ of $T$, we define the indicator variable
\[ \xi^a (e_v) = \mathbf{1}_{ \{ w_{e_v} > \B^a(v) \} }.  \]
Given any non-random boundary condition, the bonuses $\B^a(v_1), \B^a(v_2), \ldots, \B^a(v_{d_v}) $ of the children $v_1, v_2, \ldots, v_{d_v}$ of $v$
depend only on the edge weights in the subtrees  $T_{v_1}, T_{v_2}, \ldots, T_{v_{d_v}}$.  Consequently, the variables 
\begin{equation}\label{fact:bonus_indep_edge1}
 \big( w_{e_{v_1}},\B^a(v_1)\big), \big( w_{e_{v_2}},\B^a(v_2)\big), \ldots, \big( w_{e_{v_{d_v}}},\B^a(v_{d_v})\big),  \ w_{e_v}   
\end{equation}
are independent. In particular, 
\begin{equation}\label{fact:bonus_indep_edge2}
\xi^a(e_{v_1}), \ \xi^a(e_{v_2}), \ \ldots, \ \xi^a(e_{v_{d_v}}), \ w_{e_v}
\end{equation}
are independent.

From the bonus recursion \eqref{eq:tree_rec}, the memoryless property of the exponential distribution, and \eqref{fact:bonus_indep_edge1}, it follows that for any $v \notin \mathcal{L}_T \cup \{ o^\dagger \} $,
\[
\B^a(v) \stackrel{d}{=} \max_{1 \le i \le d_v} \xi^a(e_{v_i}) \, w'_{e_{v_i}},
\]
where $w'_{e_{v_i}}$ are i.i.d.\ copies of $w_{e_{v_i}}$, independent of $\xi^a(e_{v_1}), \ldots, \xi^a(e_{v_{d_v}}), w_{e_v}$.  
By \eqref{fact:bonus_indep_edge2}, we then compute
\begin{align*}
\mathbf{P}\big( \xi^a(e_v) = 1 \big) 
&= \mathbf{P} \big( w_{e_v} > \B^a(v) \big)  \\
&= \mathbf{P} \Big( w_{e_v} > \max_{1 \le i \le d_v} \xi^a(e_{v_i}) \, w'_{e_{v_i}} \Big) \\
&= \mathbf{E} \bigg[ \frac{1}{ 1 + \xi^a(e_{v_1}) + \xi^a(e_{v_2}) + \cdots + \xi^a(e_{v_{d_v}}) } \bigg].
\end{align*}

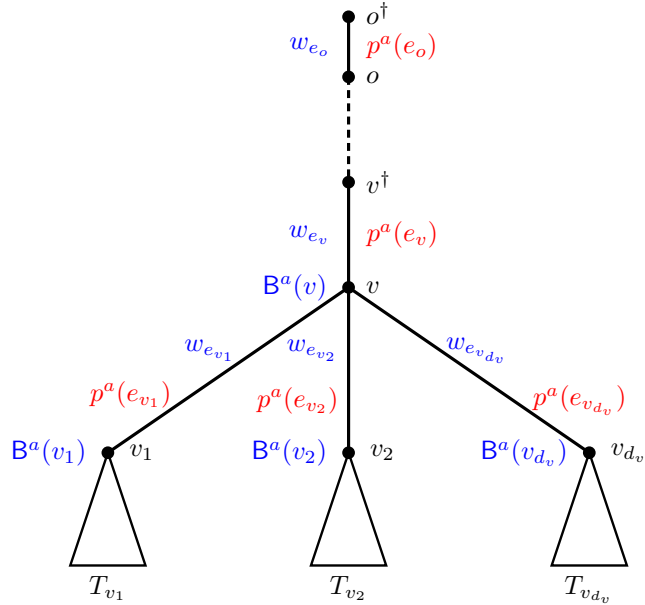
\begin{figure}[ht]
\centering
\begin{tikzpicture}[scale=1.0, every node/.style={font=\small}]
  %========================
  % Styles & parameters
  %========================
  \tikzstyle{vtx}=[circle, fill=black, inner sep=1.7pt]
  \def\edgewidth{1.2pt}

  %========================
  % Coordinates
  %========================
  \coordinate (root) at (0,3.6);
  \coordinate (o)     at (0,2.8);      % vertex o
  \coordinate (vdag)  at (0,1.4);      % vertex v^\dagger
  \coordinate (v)     at (0,0.0);      % vertex v

  % Children (spread out horizontally)
  \coordinate (v1)    at (-3.2,-2.2);
  \coordinate (v2)    at ( 0.0,-2.2);
  \coordinate (vd)    at ( 3.2,-2.2);

  %========================
  % Spine edges (thick; dashed o--v^\dagger)
  %========================
  \draw[line width=\edgewidth] (root) -- (o);
  \draw[densely dashed, line width=\edgewidth] (o) -- (vdag);
  \draw[line width=\edgewidth] (vdag) -- (v);

  %========================
  % Children edges of v (thick)
  %========================
  \draw[line width=\edgewidth] (v) -- (v1);
  \draw[line width=\edgewidth] (v) -- (v2);
  \draw[line width=\edgewidth] (v) -- (vd);

  %========================
  % Taller, narrow-base triangles for subtrees
  %========================
  \def\triW{1.0}   % base width (narrower)
  \def\triH{1.5}   % height (taller)

  % v1 triangle
  \coordinate (v1L) at ($(v1)+(-\triW/2,-\triH)$);
  \coordinate (v1R) at ($(v1)+(\triW/2,-\triH)$);
  \draw[line width=0.9pt] (v1) -- (v1L) -- (v1R) -- cycle;
  \node[anchor=north] at ($(v1L)!0.5!(v1R)$) {$T_{v_1}$};

  % v2 triangle
  \coordinate (v2L) at ($(v2)+(-\triW/2,-\triH)$);
  \coordinate (v2R) at ($(v2)+(\triW/2,-\triH)$);
  \draw[line width=0.9pt] (v2) -- (v2L) -- (v2R) -- cycle;
  \node[anchor=north] at ($(v2L)!0.5!(v2R)$) {$T_{v_2}$};

  % vd triangle
  \coordinate (vdL) at ($(vd)+(-\triW/2,-\triH)$);
  \coordinate (vdR) at ($(vd)+(\triW/2,-\triH)$);
  \draw[line width=0.9pt] (vd) -- (vdL) -- (vdR) -- cycle;
  \node[anchor=north] at ($(vdL)!0.5!(vdR)$) {$T_{v_{d_v}}$};

  %========================
  % Vertices
  %========================
  \node[vtx] at (root) {};
  \node[vtx] at (o) {};
  \node[vtx] at (vdag) {};
  \node[vtx] at (v) {};
  \node[vtx] at (v1) {};
  \node[vtx] at (v2) {};
  \node[vtx] at (vd) {};

  %========================
  % Vertex labels
  %========================
  \node[anchor=west]  at ($(root)+(0.1,0.05)$) {$o^\dagger$};
  \node[anchor=west]  at ($(o)+(0.1,0)$) {$o$};
  \node[anchor=west]  at ($(vdag)+(0.1,0)$) {$v^\dagger$};
  \node[anchor=west]  at ($(v)+(0.1,0)$) {$v$};

  \node[anchor=west] at ($(v1)+(0.15,0)$) {$v_1$};
  \node[anchor=west] at ($(v2)+(0.15,0)$) {$v_2$};
  \node[anchor=west] at ($(vd)+(0.15,0)$) {$v_{d_v}$};

  %========================
  % Bonus labels (blue) to the left of vertices
  %========================
  \node[anchor=east] at ($(v)+(-0.15,0)$) {\textcolor{blue}{$\B^a(v)$}};
  \node[anchor=east] at ($(v1)+(-0.15,0)$) {\textcolor{blue}{$\B^a(v_1)$}};
  \node[anchor=east] at ($(v2)+(-0.15,0)$) {\textcolor{blue}{$\B^a(v_2)$}};
  \node[anchor=east] at ($(vd)+(-0.15,0)$) {\textcolor{blue}{$\B^a(v_{d_v})$}};

  %========================
  % Edge labels on spine (blue/red)
  %========================
  \path (root) -- (o) node[midway, left=3pt]  {\textcolor{blue}{$w_{e_o}$}};
  \path (root) -- (o) node[midway, right=3pt] {\textcolor{red}{$p^a(e_o)$}};

  \path (vdag) -- (v) node[midway, left=3pt]  {\textcolor{blue}{$w_{e_v}$}};
  \path (vdag) -- (v) node[midway, right=3pt] {\textcolor{red}{$p^a(e_v)$}};

  %========================
  % Child edge labels: individual placement (blue left, red right)
  %========================
  \node[anchor=east] at ($(v)!0.30!(v1)+(-0.4,-0.18)$) {\textcolor{blue}{$w_{e_{v_1}}$}};
  \node[anchor=east] at ($(v)!0.60!(v1)+(-0.3,-0.1)$) {\textcolor{red}{$p^a(e_{v_1})$}};

  \node[anchor=east] at ($(v)!0.30!(v2)+(0,-0.18)$) {\textcolor{blue}{$w_{e_{v_2}}$}};
  \node[anchor=east] at ($(v)!0.60!(v2)+(0,-0.18)$) {\textcolor{red}{$p^a(e_{v_2})$}};

  \node[anchor=west] at ($(v)!0.30!(vd)+(0.2,-0.18)$) {\textcolor{blue}{$w_{e_{v_{d_v}}}$}};
  \node[anchor=west] at ($(v)!0.60!(vd)+(0.4,-0.18)$) {\textcolor{red}{$p^a(e_{v_{d_v}})$}};

\end{tikzpicture}
\caption{Illustration of the recursion  \eqref{eq:rec_p} on the edges of $T$}
\end{figure}

 This yields a recurrence relation for the means of the indicator variables.  
Letting
\[
p^a(e_v) := \mathbf{E}\big[ \xi^a(e_v) \big],
\]
we have for $v \notin \mathcal{L}_T \cup \{ o^\dagger\}$
\begin{equation}\label{eq:rec_p}
p^a(e_v)  
=  \Phi\big(p^a(e_{v_1}), \ldots, p^a(e_{v_{d_v}}) \big) 
:=  \mathbf{E} \bigg[ \frac{1}{ 1 + \xi^a(e_{v_1}) + \xi^a(e_{v_2}) + \cdots + \xi^a(e_{v_{d_v}}) } \bigg],
\end{equation}
where $\xi^a(e_{v_i}), 1 \le i \le d_v$, are independent Bernoulli random variables with means $p^a(e_{v_i})$.  For a leaf edge $e_v = (v^\dagger v)$ with $v \in \mathcal{L}_T$, we have
\begin{equation}\label{eq:rec_p_bd}
p^a(e_v) = \mathbf{P}( w_{e_v} > a_v ) = e^{-a_v}, 
\quad v \in \mathcal{L}^r_T,
\qquad\text{and}\qquad
p^a(e_v) = \mathbf{P}( w_{e_v} > 0 ) = 1, 
\quad v \in \mathcal{L}_T \setminus \mathcal{L}^r_T,
\end{equation}
which serve as the boundary conditions.

The recurrence relation \eqref{eq:rec_p}, together with the boundary conditions \eqref{eq:rec_p_bd}, can be solved iteratively by moving up the tree from the leaf edges to the root edge.  
The value at the root edge $e_o = (o^\dagger o)$, denoted $p^a(e_o)$, is thus a function of $(a_v)_{v \in \mathcal{L}^r_T}$. The following lemma is crucial, as it states that $p^a(e_o)$ becomes insensitive to the boundary condition when the height of the tree is large.

\begin{lem}\label{lem:contracttion}
Let $T$ be a tree of depth $r \ge 4$ as described in this section and assume that the maximum degree of $T$ is bounded above by $D$. Let $a' = (a_v)_{v \in \mathcal{L}_T^r}, a'' = (a''_v)_{v \in \mathcal{L}_T^r}$ be two arbitrary boundary conditions on $T$. Then 
\[ \big| p^{a'}(e_o) -  p^{a''}(e_o) \big|  \le  D\big( 1  - (2D)^{-D} \big)^{r-3}. \]
Moreover, the same bound continues to hold if we take $a'' \equiv +\infty.$
\end{lem}

We begin with a lemma that can be shown by an elementary computation. 
\begin{lem}\label{l:derivative}
Let $Y \ge 0$ be a non-negative random variable. Let  $\xi \sim \Ber(p)$, independent of $Y$. Then 
\[ \frac{\mathrm{d}}{\mathrm{d}p} \E \Big [ \frac{1}{1+ Y +\xi} \Big]  = -  \E \Big [ \frac{1}{(1+Y)(2+Y)} \Big ].   \]
\end{lem}

\begin{lem}\label{l:key_ineq}
Let $k \ge 1$. Let  $\xi_1, \xi_2, \ldots, \xi_k$ be independent Bernoulli random variables with means $p_1, p_2, \ldots, p_k$ and let $W \ge 0$ be another independent random variable. 
Let  
\[ X= \xi_1 + \cdots + \xi_k \text { and  } \quad    Z_i  = X - \xi_i =   \sum_{j: j \ne i} \xi_j \text{ for each } i.\] 
Then we have
\begin{equation} \label{eq: sum_expec1}
  \sum_{i=1}^k p_i \E \Big [ \frac{1}{(1+ W+ Z_i)(2+ W + Z_i)} \Big ] \le \E \Big [ \frac{\mathbf{1}_{\{ X>0 \}}}{1+ W+ X} \Big].    
  \end{equation}
\end{lem}
\begin{proof}
Using the independence of $\xi_i, Z_i $ and  $W,$ we can write the left-hand side of \eqref{eq: sum_expec1} as
\[    \sum_{i=1}^k  \E \Big [ \frac{\xi_i}{(1+ W+ Z_i)(2+ W+ Z_i)} \Big ].    \]
Note that $\xi_i = \xi_i \mathbf{1}_{ \{ \xi_i  = 1\} }$, and  on the event $\{\xi_i = 1 \}$,  we have $X = 1+Z_i >0$. Therefore, the above equals 
   \begin{align*}
   &   \sum_{i=1}^k  \E \Big [ \frac{\xi_i \mathbf{1}_{ \{ X>0\} }}{(W+X)(1+ W+ X)} \Big ]  =\E \Big [ \frac{X \mathbf{1}_{ \{ X>0\}}}{(W+X)(1+W+X)} \Big  ] \le \E \Big [ \frac{\mathbf{1}_{\{ X>0\}}}{1+W+X} \Big],
   \end{align*}
   where in the last inequality, we used the fact that $X \le W+X.$
\end{proof}

\begin{proof}[Proof of Lemma~\ref{lem:contracttion}]
For $k \ge 1,$ let
\[
U_k := \bigl\{\, v \in V(T) : \exists\, u \in \mathcal{L}_T^r \text{ such that } 
v \text{ lies on a path from } u \text{ to } o^\dagger \text{ and } 1\le \dist(v,o^\dagger)\le k \,\bigr\}.
\]
In particular, $U_r \setminus U_{r-1} = \mathcal{L}_T^r.$ Since $d_v \le D-1$, the recursion \eqref{eq:rec_p} ensures that irrespective of any boundary condition $a$, one has
\begin{equation} \label{q_bound_nonleaf}
  p^a(e_v) \ge \frac{1}{D},  \ \  \text{ for any } v \in U_{r-1}.  
\end{equation}
Moreover, if $v \in U_{r-2},$  then $v$ has at least one child $u \in U_{r-1}.$  Therefore,  by \eqref{q_bound_nonleaf}, we have 
\begin{equation}\label{p_upper_bd}
p^a(e_v) \le \E \!\left[ \frac{1}{1 + \xi^a(e_u)} \right] \;=\; 1 - \tfrac{1}{2} p^a(e_u) \le 1 - \tfrac{1}{2D}, \quad  \text{ for any } v \in U_{r-2}.
\end{equation}
By iterating the recursion up to two levels on the tree $T$, if necessary, we may assume that for any boundary condition $a$, the $p$-value at the root edge $p^{a}(e_o)$ is a function of the inputs $\{p^{a}(e_v): v \in U_{r-2}\setminus U_{r-3}\}$. This function is determined solely by the recursion at the vertices in $U_{r-3}$ given by~\eqref{eq:rec_q} and does not depend on the particular boundary condition $a$.

We aim to prove the lemma by establishing that the above function is contractive. Instead of working with $p$, we prove the contractivity after the reparametrization $q^a(e)=-\log p^a(e)\ge 0$. The relation \eqref{eq:rec_p} yields the following recurrence relation on $q$. For any $v \in U_{r-3}$,
\begin{equation}\label{eq:rec_q}
 q^a(e_v) = \Psi( q^a(e_{v_1}), \ldots, q^a(e_{v_{d_v}})):= - \log \Phi \big(e^{-q^a(e_{v_1})}, \ldots, e^{-q^a(e_{v_{d_v}})}\big).   \end{equation}
Fix $v\in U_{r-3}$. Let $k\ge 1$ be the number of children of $v$ that lie in $U_{r-2}\setminus U_{r-3}$, and let $\ell\ge 0$ be the number of its remaining children. Clearly $k+\ell=d_v$. Label the former by $v_1,\dots,v_k$ and the latter by $v_{k+1},\dots,v_{k+\ell}$. 

Note that the $p$-values on the edges $e_{v_i}$ for $k+1\le i\le k+\ell$ do not depend on the boundary condition $a$, since we set $\B^a(v)=0$ for all $v\in \mathcal{L}_T\setminus \mathcal{L}_T^r$. Consequently, the right-hand side of~\eqref{eq:rec_q} can be viewed as a function only of $q^a(e_{v_1}),\ldots,q^a(e_{v_k})$. The partial derivative of $\Psi(q^a(e_{v_1}), \ldots, q^a(e_{v_{k}}))$ with respect to $q^a(e_{v_i})$ for $ 1 \le i \le k$ satisfies  
\begin{equation}\label{partial_deriv_q_v}
\partial_{q^a(e_{v_i})} \Psi(q^a(e_{v_1}), \ldots, q^a(e_{v_{k+ \ell}}))
= \frac{p^a(e_{v_i})\, \partial_{p^a(e_{v_i})} \Phi(p^a(e_{v_1}), \ldots, p^a(e_{v_{k }}))}{p^a(e_{v})}.
\end{equation}
Let $\xi^a(e_{v_1}), \ldots, \xi^a(e_{v_{k+\ell}})$ be independent Bernoulli random variables with means $p^a(e_{v_1}), \ldots,$ $p^a(e_{v_{k+ \ell}})$, and define  
\[
X^a_v = \sum_{j=1}^{k} \xi^a(e_{v_j}), \quad Z^a_{v,i} = X^a_v - \xi^a(e_{v_i}), \quad 1 \le i \le k, \quad W^a_v = \sum_{j=k+1}^{k+ \ell} \xi^a(e_{v_j}).
\]
By \eqref{eq:rec_p}, \eqref{partial_deriv_q_v}, and Lemma~\ref{l:derivative},  for each $1 \le i \le k,$
\[
\partial_{q^a(e_{v_i})} \Psi(q^a(e_{v_1}), \ldots, q^a(e_{v_{d_v}}))
= -\frac{p^a(e_{v_i})}{p^a(e_{v})} \, \E\!\left[ \frac{1}{(1+W^a_v+ Z^a_{v,i})(2+W^a_v+ Z^a_{v,i})} \right] \le 0.
\]
Now denoting $\nabla \Psi(q^a(e_{v_1}), \ldots, q^a(e_{v_{k}}))$ by  the gradient 
\[
\big(\partial_{q^a(e_{v_1})} \Psi, \ldots, \partial_{q^a(e_{v_{k}})} \Psi\big) \in \mathbb{R}^{k},
\]
Lemma~\ref{l:key_ineq} yields the bound  
\begin{align*}
    \|\nabla \Psi(q^a(e_{v_1}), \ldots, q^a(e_{v_{k}}))\|_1
&= \frac{1}{p^a(e_v)} \sum_{i=1}^{k} p^a(e_{v_i})\, \E\!\left[ \frac{1}{(1+ W^a_v+ Z^a_{v,i})(2+  W^a_v+ Z^a_{v,i})} \right] \\
&\le \frac{\E \left[ \frac{\mathbf{1}_{\{X^a_v>0\}}}{1+W^a_v+ X^a_v} \right]}{\E \left[ \frac{1}{1+W^a_v + X^a_v} \right]}.
\end{align*}
Since $\mathbf{1}_{\{x>0 \}}$ is increasing and $\tfrac{1}{1+x}$ is decreasing in $x$, the correlation inequality implies that, conditioning on $W_v^a$,
 \[ \E \Big [ \frac{\mathbf{1}_{(X^a_v>0)}}{1+W^a_v + X^a_v} \Big] \le  \prob(X^a_v> 0) \E \Big [ \frac{1}{1+ W^a_v + X^a_v} \Big]. \]
Consequently, 
 \begin{align} \label{eq:rec_bound}
  \|\nabla \Psi(q^a(e_{v_1}), \ldots, q^a(e_{v_{k}}))\|_1 &\le \prob(X^a_v > 0) \nonumber \\
   &= 1 - \prod_{i=1}^{k} \big(1- p^a(e_{v_i})\big)  \nonumber \\
   &\le 1  - (2D)^{-D},
   \end{align}
where the final inequality follows from \eqref{p_upper_bd} together with the fact that $v_1, \ldots, v_k \in U_{r-2}.$
    At the root edge $e_o$, the $q$-value, $q^a(e_o)$ is a function of $\{q^{a}(e_v): v \in U_{r-2}\setminus U_{r-3}\}$, whose values lie in $[0, \log D]$ by  \eqref{q_bound_nonleaf}.  For a vertex $v \in U_{r-2}\setminus U_{r-3},$ let $v_1 = o \sim  v_1 \sim v_2 \sim \cdots \sim v_{r-2} =v$ be the unique self-avoiding path from $o$ to $v$. By the chain rule,
\[ \partial_{q^a(e_v)} q^a(e_o) = \prod_{i=1}^{r-3} \partial_{q^a(e_{v_{i+1}})} q^a(e_{v_{i}}). \]

Summing over all vertices $U_{r-2}\setminus U_{r-3}$ and then applying the gradient bound \eqref{eq:rec_bound} for $q^a(e_v)$
for vertices at level
 $k=r-3, r-4, \ldots, 1$, i.e.,  vertices in $U_k \setminus U_{k-1}$, and proceeding up the tree level by level, we obtain  
 \[  \| \nabla q^a(e_o) \|_1 \le \big( 1  - (2D)^{-D} \big)^{r-3}.\]
Therefore, by mean value theorem, 
   \begin{align*}
   \big|  q^{a'}(e_o)  - q^{a''}(e_o)  \big| 
   &\le \big( 1  - (2D)^{-D} \big)^{r-3}  \max_{v \in U_{r-2} \setminus U_{r-3}} \big | q^{a'}(e_v)  - q^{a''}(e_v)\big|\\
   &\le \log D  \big( 1  - (2D)^{-D} \big)^{r-3}  \le D \big( 1  - (2D)^{-D} \big)^{r-3}.
   \end{align*}
The lemma follows from the observation that
   \begin{align*}
   \big|  p^{a'}(e_o)  - p^{a''}(e_o)  \big|  \le    \big|  q^{a'}(e_o)  - q^{a''}(e_o)  \big| .
   \end{align*}
   
Finally, the case $a'' \equiv +\infty$ follows from the monotone convergence theorem, together with the monotonicity property \eqref{anti_monotone} of the bonuses with respect to the boundary condition.
\end{proof}

\subsubsection{Proof of Theorem~\ref{thm: temp_corr_decay}}
We now prove Theorem~\ref{thm: temp_corr_decay}. Suppose that $T := \bB_e^r(G)$ is a tree. Let $T_{v \to u}$ and $T_{u \to v}$ be the subtrees rooted at $u$ and $v$, respectively, after we remove the edge $e =(uv)$ from $T$. Note that each of these two trees has depth at most $r$, and at least one has depth exactly $r$.
By \eqref{eq: antimonotone_stronger},  for any 
$A\subseteq \partial T,$
\begin{align*}
        \mathbf{1}_{\{w_{(uv)} >  \B_r^{\infty}(u, T_{v \to u}) + \B_r^{\infty}(v, T_{u \to v})\}} &\le  \mathbf{1}_{\{e \in \M_{T\setminus A}\}} \le \mathbf{1}_{\{w_{(uv)} >  \B_r^{0}(u, T_{v \to u}) + \B_r^{0}(v, T_{u \to v})\}} \quad & \text{ if $r$ is even,}\\
                \mathbf{1}_{\{w_{(uv)} >  \B_r^{0}(u, T_{v \to u}) + \B_r^{0}(v, T_{u \to v})\}} &\le  \mathbf{1}_{\{e \in \M_{T\setminus A}\}} \le \mathbf{1}_{\{w_{(uv)} >  \B_r^{\infty}(u, T_{v \to u}) + \B_r^{\infty}(v, T_{u \to v})\}} \quad &\text{ if $r$ is odd}.
    \end{align*}
This implies that
\begin{align*}
        &\sup_{A\subseteq \partial T} |\mathbf{1}_{\{e\in \M_{T}\}} - \mathbf{1}_{\{e \in \M_{T\setminus A}\}}| \\
        \leq&\;  (-1)^{r+1}\Big(\mathbf{1}_{\{w_{(uv)} >  \B_r^{\infty}(u, T_{v \to u}) + \B_r^{\infty}(v, T_{u \to v})\}}- \mathbf{1}_{\{w_{(uv)} >  \B_r^{0}(u, T_{v \to u}) + \B_r^{0}(v, T_{u \to v})\}}\Big).
\end{align*}
Taking expectations on both sides, we obtain 
\begin{align}
        &\E \sup_{A \subseteq \partial T} |\mathbf{1}_{\{e\in \M_{T}\}} - \mathbf{1}_{\{e \in \M_{T\setminus A}\}}| \nonumber\\
        \label{eq: difference_prob_2}
        \leq \;& \Big| \prob(w_{(uv)} >  \B_r^{\infty}(u, T_{v \to u}) + \B_r^{\infty}(v, T_{u \to v})) - \prob(w_{(uv)} >  \B_r^{0}(u, T_{v \to u}) + \B_r^{0}(v, T_{u \to v})) \Big|.
\end{align}
Note that $ w_{(uv)}, \B_r^{\infty}(u, T_{v \to u}), \B_r^{\infty}(v, T_{u \to v}) $ are independent due to edge disjointness; the same holds for  $ w_{(uv)}, \B_r^{0}(u, T_{v \to u}), \B_r^{0}(v, T_{u \to v})$. Thus, by the memoryless property of the exponential distribution, we have
\begin{align*}
\prob\big(w_{(uv)} >  \B_r^{\infty}(u, T_{v \to u}) + \B_r^{\infty}(v, T_{u \to v})\big) &= \prob \big(w_{(uv)} >  \B_r^{\infty}(u, T_{v \to u})\big) \prob \big(w_{(uv)} > \B_r^{\infty}(v, T_{u \to v})\big),\\
\prob\big (w_{(uv)} >  \B_r^{0}(u, T_{v \to u}) + \B_r^{0}(v, T_{u \to v})\big) &= \prob\big (w_{(uv)} >  \B_r^{0}(u, T_{v \to u})\big) \prob\big (w_{(uv)} > \B_r^{0}(v, T_{u \to v})\big).
\end{align*}

Let $T^+_{v \to u}$ and $T^+_{u \to v}$ be the subtrees of $T$ formed by adding the edge $e = (uv)$ to the trees $T_{v \to u}$ and $T_{u \to v},$ respectively, and relocating the roots to $v$ and $u$, respectively. Using the notation of Lemma~\ref{lem:contracttion}, we define the following probabilities on the edge $e$ in the tree  $T^+_{v \to u},$
\[p^\infty_{e}(T^+_{v \to u}) := \prob(w_{(uv)} >  \B_r^{\infty}(u, T_{v \to u})), \quad p^0_{e}(T^+_{v \to u}) := \prob(w_{(uv)} >  \B_r^{0}(u, T_{v \to u})). \]
Similarly, for the tree $T^+_{u \to v},$ we define the probabilities 
$p^\infty_{e}(T^+_{u \to v}) $ and $p^0_{e}(T^+_{u \to v})$. We can then express and bound \eqref{eq: difference_prob_2} as 
\begin{align}\label{eq:pfactor}
   \big| p^\infty_{e}(T^+_{v \to u}) p^\infty_{e}(T^+_{u \to v}) - p^0_{e}(T^+_{v \to u})p^0_{e}(T^+_{u \to v}) \big|
   \le | p^\infty_{e}(T^+_{v \to u}) - p^0_{e}(T^+_{v \to u})| + | p^\infty_{e}(T^+_{u \to v}) - p^0_{e}(T^+_{u \to v})|.
\end{align}
If the  tree $T_{v \to u}$  has depth strictly less than $r$, then  $\B_r^{\infty}(u, T_{v \to u}) =  \B_r^{0}(u, T_{v \to u})$ and hence, $p^\infty_{e}(T^+_{v \to u})  =  p^0_{e}(T^+_{v \to u}).$ The same conclusion holds for the tree $T_{u \to v}.$ Therefore, we may assume the nontrivial case in which both have depth $r$; in that case, both $T^+_{v \to u}$ and $T^+_{u \to v}$ have depth $r+1$.

To bound the the right-hand side of \eqref{eq:pfactor}, we apply Lemma~\ref{lem:contracttion} twice for the trees $T^+_{v \to u}$ and $T^+_{u \to v}$ (with boundary conditions $a' = 0$ and $a'' = \infty$)   to obtain
\[
\E\sup_{A\subseteq \partial T} |\mathbf{1}_{\{e\in \M_{T}\}} - \mathbf{1}_{\{e \in \M_{T\setminus A}\}}|\leq 2D\big( 1  - (2D)^{-D} \big)^{r-2}.
\]
This completes the proof of Theorem~\ref{thm: temp_corr_decay}.

\subsection{Proof of Theorem~\ref{thm: corr_decay_degree_3}}
We start with a simple lemma.
\begin{lem} \label{lem:exp_indentity}
    Let $X_1, X_2, \ldots, X_k$ be i.i.d.\ exponentials with mean $1$. For $c_1, c_2, \ldots, c_k \ge 0$, we have 
    \[ \E[ e^{-\max_{1 \le i \le k}  (X_i -c_i)_+}] =  \sum_{S \subseteq [k]} \frac{1}{1+ |S|} \prod_{i \in S} e^{-c_i} \prod_{i \not \in S} (1-e^{-c_i}). \]
\end{lem}
\begin{proof}
Let $ X_1', X_2', \ldots, X_k' $ be i.i.d.\ exponential random variables with mean 1, independent of $ X_1, X_2, \ldots, X_k$. By the memoryless property of the exponential distribution, we can express the left-hand side as
   \begin{align*}
    &\sum_{S \subseteq [k]}   \E[ e^{-\max_{1 \le i \le k}  (X_i -c_i)_+} \mathbf{1}_{\{ X_i > c_i \;\forall i \in S, \ \ X_i \le c_i  \;\forall i \not \in S^c \}}]  \\
    =& \sum_{S \subseteq [k]}   \E[ e^{-\max_{i \in S}  X_i'} \mathbf{1}_{\{ X_i > c_i \;\forall i \in S, \ \ X_i \le c_i  \;\forall i \not \in S^c \}}] \\
     =& \sum_{S \subseteq [k]}   \E[ e^{-\max_{i \in S}  X_i'}]  \prod_{i \in S} e^{-c_i} \prod_{i \not \in S} (1-e^{-c_i}).
    \end{align*}
   An elementary calculation shows that $\E[ e^{-\max_{i \in S}  X_i'}] = (1+ |S|)^{-1}$, and the lemma follows. 
\end{proof}

We now continue with the proof of Theorem~\ref{thm: corr_decay_degree_3}. Fix a subgraph $H$ of $G$ with $|V(H)| \ge 1$. For any vertex $u \in V(H)$, let $d_u = d_u(H)$ denote the degree of $u$ in $H$. Let $ u_1, \ldots, u_{d_u}$ denote the set of neighbors of $u$. Note that although the bonuses $\B(u_i, H \setminus \{u\})$ for $1 \leq i \leq d_u$ may not be mutually independent, they are independent of the edge weights $w_{(uu_i)}$ for $1 \leq i \leq d_u$.
Using the bonus recursion \eqref{eq: bonus_recursion} and Lemma~\ref{lem:exp_indentity}, it follows that 
\begin{align}\label{eq:recursion1}
    \E[e^{-\B(u, H)} \mid \B(u_i, H\setminus\{u\}), 1\leq i\leq d_u] 
    =& \sum_{S\subseteq [d_u]} \frac{1}{1+|S|} \prod_{i\in S} e^{-\B(u_i, H\setminus\{u\})}\prod_{i\not\in S} (1 - e^{-\B(u_i, H\setminus\{u\})}).
\end{align}

Let us now assume that  $d_u$ is at most $2$. 
When $d_u = 2$, \eqref{eq:recursion1} becomes
\begin{align}
    &\E[e^{-\B(u, H)} \mid \B(u_1, H\setminus\{u\}), \B(u_2, H\setminus\{u\})] \nonumber\\
    \label{eq: degree_2}
    =&\; 1 - \frac{1}{2}(e^{-\B(u_1, H\setminus\{u\})} + e^{-\B(u_2, H\setminus\{u\})}) + \frac{1}{3} e^{-(\B(u_1, H\setminus\{u\}) + \B(u_2, H\setminus\{u\}))},
\end{align}
whereas when $d_u = 1$, we have
\begin{align}
    \label{eq: degree_1}
    \E[e^{-\B(u, H)} \mid \B(u_1, H\setminus\{u\})] = 1 - \frac{1}{2}e^{-\B(u_1, H\setminus\{u\})}.
\end{align}

From now on, we assume that $r$ is a fixed even integer (the case where $r$ is odd is analogous). If $d_u = 1$, we write $u_1$ for the unique neighbor of $u$ in $H$; if $d_u = 2$, we write $u_1$ and $u_2$ for the two neighboring vertices of $u$. Recall the local bounds in Section~\ref{sec: local_bounds}. To simplify the notations, we will write 
\begin{align*}
    \B_r^\infty &= \B_r^\infty(u, H), & \B_r^0 &= \B_r^0(u, H),&\\
    \B_{r-1}^\infty(i) &= \B_{r-1}^\infty(u_i, H \setminus \{u\}), & \B_{r-1}^0(i) &= \B_{r-1}^0(u_i, H \setminus \{u\}).&
\end{align*}
Since $r$ is even, we have
\[
\B_r^\infty \geq \B_r^0, \quad \B_{r-1}^\infty(i) \leq \B_{r-1}^0(i).
\]
When $d_u = 1$, it follows from \eqref{eq: degree_1} that
\begin{align}\label{bdd:d_u=1}
0 \le \E e^{-\B_r^0} - \E e^{-\B_r^\infty} \leq \frac{1}{2}\big( \E e^{-\B_{r-1}^\infty(1)} - \E e^{-\B_{r-1}^0(1)} \big).
\end{align}
The case $d_u = 2$ is slightly more complicated. By \eqref{eq: degree_2}, we have
\begin{align}
\label{eq: long_eq}
\begin{split}
    \E e^{-\B_r^0} - \E e^{-\B_r^\infty} =&\; \frac{1}{2} [(\E e^{-\B_{r-1}^\infty(1)} - \E e^{-\B_{r-1}^0(1)}) + (\E e^{-\B_{r-1}^\infty(2)} - \E e^{-\B_{r-1}^0(2)})] \\
    &+ \frac{1}{3} \E[e^{-(\B_{r-1}^0(1) + \B_{r-1}^0(2))} - e^{-(\B_{r-1}^\infty(1) + \B_{r-1}^\infty(2))}].
\end{split}    
\end{align}
Note that
\begin{align*}
e^{-(\B_{r-1}^0(1) + \B_{r-1}^0(2))} - e^{-(\B_{r-1}^\infty(1) + \B_{r-1}^\infty(2))} =&\; (e^{-\B_{r-1}^0(1)} - e^{-\B_{r-1}^\infty(1)})e^{-\B_{r-1}^0(2)} \\
&+ (e^{-\B_{r-1}^0(2)} - e^{-\B_{r-1}^\infty(2)}) e^{-\B_{r-1}^\infty(1)}.
\end{align*}
Putting this back into \eqref{eq: long_eq}, we have
\begin{align}
\label{eq: first_term}
\E e^{-\B_r^0} - \E e^{-\B_r^\infty} =&\; \frac{1}{2}\E\left[\left(1 - \frac{2}{3} e^{-\B_{r-1}^0(2)}\right)\left(e^{-\B_{r-1}^\infty(1)} - e^{-\B_{r-1}^0(1)}\right)\right]\\
\label{eq: second_term}
&+ \frac{1}{2}\E\left[\left(1 - \frac{2}{3} e^{-\B_{r-1}^\infty(1)}\right)\left(e^{-\B_{r-1}^\infty(2)} - e^{-\B_{r-1}^0(2)}\right)\right].
\end{align}

We first handle \eqref{eq: first_term}. Fix $\varepsilon \in (0, 1)$. Then
\begin{align*}
    & \E\left[\left(1 - \frac{2}{3} e^{-\B_{r-1}^0(2)}\right)\left(e^{-\B_{r-1}^\infty(1)} - e^{-\B_{r-1}^0(1)}\right)\right]\\
    =&\; \E\left[\left(1 - \frac{2}{3} e^{-\B_{r-1}^0(2)}\right)\left(e^{-\B_{r-1}^\infty(1)} - e^{-\B_{r-1}^0(1)}\right)\mathbf{1}_{\{e^{-\B_{r-1}^0(2)} > \frac{3}{2}\varepsilon\}}\right] \\
    &+ \E\left[\left(1 - \frac{2}{3} e^{-\B_{r-1}^0(2)}\right)\left(e^{-\B_{r-1}^\infty(1)} - e^{-\B_{r-1}^0(1)}\right)\mathbf{1}_{\{e^{-\B_{r-1}^0(2)} \leq \frac{3}{2}\varepsilon\}}\right]\\
    \le&\; (1-\varepsilon)\E\left[\left(e^{-\B_{r-1}^\infty(1)} - e^{-\B_{r-1}^0(1)}\right)\mathbf{1}_{\{e^{-\B_{r-1}^0(2)} > \frac{3}{2}\varepsilon\}}\right] + \E\left[\left(e^{-\B_{r-1}^\infty(1)} - e^{-\B_{r-1}^0(1)}\right)\mathbf{1}_{\{e^{-\B_{r-1}^0(2)} \leq \frac{3}{2}\varepsilon\}}\right]\\
    =&\;(1-\varepsilon)\E\left[e^{-\B_{r-1}^\infty(1)} - e^{-\B_{r-1}^0(1)}\right] + \varepsilon \E\left[\left(e^{-\B_{r-1}^\infty(1)} - e^{-\B_{r-1}^0(1)}\right)\mathbf{1}_{\{e^{-\B_{r-1}^0(2)} \leq \frac{3}{2}\varepsilon\}}\right]\\
    \leq &\; (1-\varepsilon)\E\left[e^{-\B_{r-1}^\infty(1)} - e^{-\B_{r-1}^0(1)}\right] + \varepsilon\prob\left(e^{-\B_{r-1}^0(2)} \leq \frac{3}{2}\varepsilon\right).
\end{align*}
Note that $\B_{r-1}^0(2)$ is stochastically dominated by $\max (w_1, w_2)$, where $w_1$ and $w_2$ are i.i.d.\ exponential random variables with parameter $1$. So
\begin{align*}
    \prob\left(e^{-\B_{r-1}^0(2)} \leq \frac{3}{2}\varepsilon\right) &= \prob\left(\B_{r-1}^0(2) \geq \log\frac{2}{3\varepsilon}\right)\\
    &\leq \prob\left(\max (w_1, w_2) \geq \log\frac{2}{3\varepsilon}\right)\\
    & \leq 2\prob\left(w_1 \geq \log\frac{2}{3\varepsilon}\right) = 3\varepsilon.
\end{align*}
Hence, we have
\[
\E\left[\left(1 - \frac{2}{3} e^{-\B_{r-1}^0(2)}\right)\left(e^{-\B_{r-1}^\infty(1)} - e^{-\B_{r-1}^0(1)}\right)\right] \leq (1-\varepsilon)\E\left[e^{-\B_{r-1}^\infty(1)} - e^{-\B_{r-1}^0(1)}\right] + 3\varepsilon^2.
\]
Similarly, \eqref{eq: second_term} can be bounded by
\[
\E\left[\left(1 - \frac{2}{3} e^{-\B_{r-1}^\infty(1)}\right)\left(e^{-\B_{r-1}^\infty(2)} - e^{-\B_{r-1}^0(2)}\right)\right] \leq (1-\varepsilon)\E\left[e^{-\B_{r-1}^\infty(2)} - e^{-\B_{r-1}^0(2)}\right] + 3\varepsilon^2.
\]
Thus, in the case $d_u=2$, we have, for all $\varepsilon \in (0, 1)$,
\begin{align}\label{bdd:d_u=2}
   |\E e^{-\B_r^0} - \E e^{-\B_r^\infty}| \leq (1-\varepsilon)\max_{i=1,2}|\E e^{-\B_{r-1}^0(i)} - \E e^{-\B_{r-1}^\infty(i)}| + 3\varepsilon^2. 
\end{align}

Let
\[
M_r = \max_{ H } \max_{u\in V(H)}|\E e^{-\B_r^0(u, H)} - \E e^{-\B_r^\infty(u, H)}|,
\]
where the maximum is taken over all subgraphs $H$ of $G$ and $u \in V(H)$ such that $d_u \le 2$. Since the degree of $u_i$ in $H \setminus \{u\}$ is at most $2$, from \eqref{bdd:d_u=1} and \eqref{bdd:d_u=2}, we  have the following recursion
\[
M_r \leq (1-\varepsilon)M_{r-1} + 3\varepsilon^2, \quad M_0 =  1,
\]
which holds for all $\varepsilon \in (0, 1).$
From this, one can deduce that for all $r \ge 1,$
\begin{align} \label{eq:estimate_degree_max3}
M_r \leq (1-\varepsilon)^r + 3\varepsilon.
\end{align}

Now, let $e=(uv)$ be an edge of $G$ and set 
\[ H' = \bB_e^r(G)  \setminus \{ e \}, \quad H'' =  \bB_e^r(G)  \setminus \{u\}. \]
Let $A \subseteq \partial \bB_e^r(G)$. Using the criterion \eqref{eq: equivalence}, we have 
\begin{align}\label{eq:bd1}
\big| \mathbf{1}_{ \{ e \in \M_{\bB_e^r(G)} \}} - \mathbf{1}_{ \{ e \in \M_{\bB_e^r(G) \setminus A} \}}\big|
&= \big| \mathbf{1}_{ \{ w_e >\B(u, H') + \B(v, H'')\}} - \mathbf{1}_{ \{ w_e > \B(u, H' \setminus A) + \B(v, H'' \setminus A)\}}\big|.
\end{align}
 The vertices of $\partial \bB_e^r(G)$ in $H'$ (respectively $H''$) are at distance $r$ or $r+1$ from $u$ in $H'$ (respectively from $v$ in $H''$).
Therefore, by \eqref{eq: antimonotone_stronger}, for any $A \subseteq \partial \bB_e^r(G)$, we have 
\begin{align*}
 (-1)^r \B_r^0(u, H') \le    (-1)^r \B(u, H' \setminus A) \le (-1)^r \B_r^\infty(u, H'), \\
  (-1)^r \B_r^0(v, H'') \le    (-1)^r \B(v, H'' \setminus A) \le (-1)^r \B_r^\infty(v, H'').
\end{align*}
Therefore, from \eqref{eq:bd1}, we derive that 
\begin{align*}
 \sup_{A \subseteq \partial \bB_e^r(G)} \big| \mathbf{1}_{ \{ e \in \M_{\bB_e^r(G)} \}} - \mathbf{1}_{ \{ e \in \M_{\bB_e^r(G) \setminus A} \}}\big|  \le (-1)^{r+1}  \big ( \mathbf{1}_{ \{ w_e >\B_r^0(u, H') + \B_r^0(v, H'')\}} - \mathbf{1}_{ \{ w_e > \B_r^\infty(u, H') + \B_r^\infty(v, H'')  \}} \big).
\end{align*}
Taking expectations on both sides of the inequality above, and noting that $w_e$ is independent of the weights used in the bonuses on the right-hand side, we obtain
\begin{align*}
\E \sup_{A \subseteq \partial \bB_e^r(G)} &\big| \mathbf{1}_{ \{ e \in \M_{\bB_e^r(G)} \}} - \mathbf{1}_{ \{ e \in \M_{\bB_e^r(G) \setminus A} \}}\big|  \le \big | \E[ e^{-\B_r^0(u, H')} e^{-\B_r^0(v, H'')}] - \E[e^{-\B_r^\infty(u, H')} e^{- \B_r^\infty(v, H'')} ] \big| \\
&\le \big | \E[ e^{-\B_r^0(u, H')} ] - \E[e^{-\B_r^\infty(u, H')}  ] \big| + 
\big | \E[ e^{-\B_r^0(v, H'')}] - \E[ e^{- \B_r^\infty(v, H'')} ] \big| \le 2M_r,
\end{align*}
where the second step above follows from the triangle inequality. The proof of the theorem now follows directly from the bound in \eqref{eq:estimate_degree_max3}.

\section{Maximum weight matching on an infinite graph}\label{sec_mwm_infinite}
\subsection{Existence and basic properties}
In this section, we prove Proposition~\ref{thm: MWM_inf_graph}.

\begin{proof}
\noindent (a)
 Define 
 \[ X_k = \sup_{m, n \ge k} |\mathbf{1}_{\{e \in \M_{G_n}\}} - \mathbf{1}_{\{e \in \M_{G_m}\}}|.\]
Fix $ r \ge 1.$ Since $(G_n)$ exhausts $G$, there exists $n_r \ge 1$  large such that 
    $E(\bB_e^r(G)) \subseteq E(G_n)$ for all $n \ge n_r.$
   Let $n \ge n_r$. Consider the MWM $\M_{G_n}$ on $G_n$, and let $\mathsf{A}_n$ be the set of all vertices in $\partial \bB_e^r(G_n)$ that are matched in $\M_{G_n}$ with some vertex in $V(G_n) \setminus V(\bB_e^r(G))$. By Observation~\ref{obs: forbidden_vertices} 
    \begin{equation}
    %\label{eq: claim_indicators}
    \mathbf{1}_{\{e \in \M_{\bB_e^r(G) \setminus \mathsf{A}_n}\}} = \mathbf{1}_{\{e \in \M_{G_n}\}} \quad \text{almost surely.}
    \end{equation}
Therefore, for $m, n \ge n_r,$ we have
    \begin{equation}\label{ineq:triangle_G_n}
          |\mathbf{1}_{\{e \in \M_{G_n}\}} - \mathbf{1}_{\{e \in \M_{G_m}\}}| \leq 2\sup_{A\subseteq \partial \bB_e^{r}(G)} |\mathbf{1}_{\{e \in \M_{\bB_e^{r}(G)}\}} - \mathbf{1}_{\{e \in \M_{\bB_e^{r}(G) \setminus A}\}}|. 
    \end{equation}
    Hence,
    \[ X_{n_r} \le 2\sup_{A\subseteq \partial \bB_e^{r}(G)} |\mathbf{1}_{\{e \in \M_{\bB_e^{r}(G)}\}} - \mathbf{1}_{\{e \in \M_{\bB_e^{r}(G) \setminus A}\}}|, \]
    which converges to $0$ in probability as $r\to \infty$ from the correlation decay assumption.
    Since $(X_k)$ is non-increasing, it converges almost surely.  Thus, the sequence $(\mathbf{1}_{\{e \in \M_{G_n}\}})$ is almost surely Cauchy, which implies that
    \[
    \lim_{n\to\infty} \mathbf{1}_{\{e\in \M_{G_n}\}} \quad \text{exists almost surely.}
    \]

    \noindent (b)   Fix $e \in E(G)$, and let $(G_n)$ and $(G_n')$ be two sequences of finite graphs increasing to $G$. For each $r$, choose  $n_r$ sufficiently large be  such that 
    $E(\bB_e^r(G)) \subseteq E(G_{n_r}) \cap E(G'_{n_r}) $. Similar to \eqref{ineq:triangle_G_n}, we have 
    \[  |\mathbf{1}_{\{e \in \M_{G_{n_r}}\}} - \mathbf{1}_{\{e \in \M_{G'_{n_r}}\}}| \leq 2\sup_{A\subseteq \partial \bB_e^{r}(G)} |\mathbf{1}_{\{e \in \M_{\bB_e^{r}(G)}\}} - \mathbf{1}_{\{e \in \M_{\bB_e^{r}(G) \setminus A}\}}|, \]
    which converges to $0$ in probability as $r \to \infty$ by correlation decay assumption. Consequently, 
    \[
     \lim_{n\to\infty} \mathbf{1}_{\{e \in \M_{G'_{n}}\}} = \lim_{r\to\infty} \mathbf{1}_{\{e \in \M_{G'_{n_r}}\}} = \lim_{r\to\infty} \mathbf{1}_{\{e \in \M_{G_{n_r}}\}} = \mathbf{1}_{\{e \in \M_G\}} \quad \text{almost surely.}
    \]
    
    \noindent (c) This is clear from the construction of $\M_G$.

    \noindent (d)   Let $(G_n)$ be a sequence of finite graphs that increases to $G$. By part (a),  for each $e \in E(G),$
    \[
    \lim_{n\to\infty} \mathbf{1}_{\{e \in \M_{G_n}\}} = \mathbf{1}_{\{e \in \M_G\}}  \quad \text{almost surely.}
    \]
    Fix a realization of the weights on this full probability event. Since $M\triangle \M_G$ is finite, there exists $n_0$ such that
$E(M\triangle \M_G)\subseteq E(G_n)$ for all $n\ge n_0$. Choose $ m \ge n_0$ sufficiently large that 
     \[ \mathbf{1}_{\{e \in \M_{G_m}\}} = \mathbf{1}_{\{e \in \M_G\}}  \quad \text{ for all } e \in E(M \triangle \M_G).
    \]
    Since $\M_G$ and $M$ agree outside $G_m$, we have 
 $\M_G \triangle M =  \M_{G_m} \triangle M_m,$ where 
$M_m$  is the restriction of $M$ to the finite graph $G_m$. Consequently, 
    \begin{equation}\label{eq:sum_alternating_local}
       \sum_{ e \in \M_{G} \setminus M}  w_e  -  \sum_{ e \in M \setminus \M_{G}} w_e = \sum_{ e \in \M_{G_m} \setminus M_m}  w_e  -  \sum_{ e \in M_m \setminus \M_{G_m}} w_e. 
    \end{equation}
Note that $\M_{G_m}\triangle M_m$ is a disjoint union of finite paths (or cycles) with edges alternating between $\M_{G_m}$ and $M_m$. If the right side of \eqref{eq:sum_alternating_local} were negative, this would yield an augmenting path for $\M_{G_m}$, contradicting that $\M_{G_m}$ is a MWM (Lemma~\ref{lem: augmenting}). Hence both sides of \eqref{eq:sum_alternating_local} are nonnegative, and they are strictly positive under generic weights.
\end{proof}

\subsection{Continuity of MWM with respect to local weak convergence}
Next, we prove Proposition~\ref{lem: local_weak_convergence_matching}.

\begin{proof}
    By  Lemma~\ref{lem: local_weak_convergence}, $(G_n, o_n, (w_e)_{e\in E(G_n)})$ converges locally weakly to $(G, o, (w_e)_{e\in E(G)})$ in $\mathcal{G}^*_{\R}$. Since $\{0, 1\}$ is compact, the measures  $(G_n, o_n, (w_e)_{e\in E(G_n)}, (\mathbf{1}_{\{e \in \M_{G_n}\}})_{e \in E(G_n)})$ is tight on $\mathcal{G}^*_{\mathbb{R}^2}$ and there exists a subsequence (which we will still index by $n$ for notational convenience) along which   $(G_n, o_n, (w_e)_{e\in E(G_n)}, (\mathbf{1}_{\{e \in \M_{G_n}\}})_{e \in E(G_n)})$ converges to $(G, o, (w_e)_{e\in E(G)}, (\mathbf{1}_{\{e \in M\}})_{e \in E(G)})$ locally weakly in $\mathcal{G}^*_{\R^2}.$
It remains to show that $(G,o,(w_e)_{e\in E(G)},M)$ and $(G,o,(w_e)_{e\in E(G)},\M_G)$ have the same distribution.

By the Skorokhod representation theorem, we can construct some probability space on which $\big(G_n, o_n, (w_e)_{e\in E(G_n)}, (\mathbf{1}_{\{e \in \M_{G_n}\}})_{e \in E(G_n)}\big)$ converges to $\big(G, o, (w_e)_{e\in E(G)}, (\mathbf{1}_{\{e \in M\}})_{e \in E(G)}\big)$  in the metric space $\big( \mathcal{G}^*_{\mathbb{R}^2}, d_{\mathrm{loc}}^w)$ almost surely. 

Fix an edge $e=(uv) \in E(G)$ and $r \ge 1.$  Choose $R$ large such $\bB_{e}^{r+1}(G) \subseteq \bB_{o}^{R}(G)$. For all $n \ge n_R$ sufficiently large, there exists a rooted isomorphism $\phi_n^{R}$ from $\bB_{o}^{R}(G)$ to $\bB_{o_n}^{R}(G_n)$ such that almost surely
    \begin{equation}\label{eq:local_weak_cons1}
    \bB_{o_n}^{R}(G_n) = \phi_n^{R}( \bB_o^{R}(G)), \quad  \M_{G_n} \cap E(\bB_{o_n}^{R}(G_n)) = \phi_n^{R}\big(M \cap E(\bB_o^{R}(G))\big)
    \end{equation}
    and furthermore, 
    \[
    w_{\phi_n^{R}(e)} \to w_e \quad \text{for all $e \in E(\bB_o^{R}(G))$ as }  \ n \to \infty.
    \]
For $n \ge n_R$, let $e_n  = \phi_n^R(e) \in E(G_n)$ and define  $\mathsf{A}^r_{n} \subseteq \partial \bB_{e_n}^r(G_n)$ to be the set of vertices $x$ for which there exists $y$ with $(xy)\in \M_{G_n}$ but $(xy)\notin E(\bB_{e_n}^r(G_n))$. Then thanks to \eqref{eq:local_weak_cons1},  $\mathsf{A}^r_n$ stabilizes to some set $\mathsf{A}^r \subseteq \partial \bB_{e}^r(G)$, in the sense that almost surely, 
    \begin{equation}\label{eq:local_weak_cons2}
        \mathsf{A}^r_{n}  = \phi_n^{R}(\mathsf{A}^r)  \quad \text{for all $n \ge n_R$.}        
        \end{equation}
 It follows from Observation~\ref{obs: forbidden_vertices} that
    \begin{align*}
        \mathbf{1}_{\{e_n \in \M_{G_n}\}} &= \mathbf{1}_{\{ e_n \in \M_{\bB_{e_n}^r(G_n) \setminus \mathsf{A}^r_{n} }\}}.
    \end{align*}
  The left side  converges to $\mathbf{1}_{\{e \in M\}}$ almost surely  by  \eqref{eq:local_weak_cons1} and  the right side converges to $\mathbf{1}_{\{ e \in \M_{\bB_{e}^r(G) \setminus \mathsf{A}^r}\}}$ almost surely by  \eqref{eq:local_weak_cons1} and \eqref{eq:local_weak_cons2}.  Consequently,
  we have 
  \[ \mathbf{1}_{\{e\in M\}} = \mathbf{1}_{\{ e \in \M_{\bB_{e}^r(G) \setminus \mathsf{A}^r}\}} \quad \text{almost surely,}  \]
  which holds for each $r \ge 1.$

  On the other hand, since $G$ satisfies the correlation decay assumption, under the law of $(G, o)$, almost surely, 
    \[
    \E_w \big |\mathbf{1}_{\{e \in \M_{\bB_{e}^r(G)}\}} - \mathbf{1}_{\{e \in \M_{\bB_{e}^r(G) \setminus \mathsf{A}_r}\}}\big| \to 0 \quad \text{as $r\to\infty$,}
    \]
    where $\E_w$ is the expectation with respect to the random weights. Thus, there exists a sequence $(r_k)$ with $r_k \to\infty$ such that almost surely,
    \[
    \big|\mathbf{1}_{\{e \in \M_{\bB_{e}^{r_k}(G)}\}} - \mathbf{1}_{\{ e \in \M_{\bB_{e}^{r_k}(G) \setminus \mathsf{A}_r}\}} \big| \to 0 \quad \text{as $k\to\infty$.}
    \]
    Furthermore, by Proposition~\ref{thm: MWM_inf_graph}(b),  $\mathbf{1}_{\{e \in \M_{\bB_{e}^{r_k}(G)}\}} \to \mathbf{1}_{\{e \in \M_G\}}$ almost surely as $k \to \infty$. Combining these observations, we obtain that 
    \[
    \mathbf{1}_{\{ e\in M\}} = \mathbf{1}_{\{e\in \M_G\}} \ \ \text{almost surely,}
    \]
    implying that $M = \M_G$ almost surely. This proves the proposition.
\end{proof}

Finally, we prove Corollary~\ref{thm: lln}.  By Proposition~\ref{lem: local_weak_convergence_matching}, $\big (G_n, o_n, (w_e)_{e\in E(G_n)}, (\mathbf{1}_{\{e \in \M_{G_n}\}})_{e \in E(G_n)} \big )$ converges locally weakly to $\big (G, o, (w_e)_{e\in E(G)}, (\mathbf{1}_{\{e \in \M_{G}\}})_{e \in E(G)}\big )$ in $\big( \mathcal{G}^*_{\mathbb{R}^2}, d^w_{\mathrm{loc}}\big)$. 

Note that
\begin{align*}
	\E \frac{\W_{G_n}}{|V(G_n)|} = \E \frac{\sum_{e \in E(G_n)} w_{e}\mathbf{1}_{\{e \in \M_{G_n}\}}}{|V(G_n)|}
	= \frac{1}{2}\E \sum_{v: v\sim o_n} w_{(o_nv)} \mathbf{1}_{\{(o_nv) \in \M_{G_n}\}}
\end{align*}
and
\begin{align*}
	\E \frac{|\M_{G_n}|}{|V(G_n)|} = \E \frac{\sum_{e \in E(G_n)} \mathbf{1}_{\{e \in \M_{G_n}\}}}{|V(G_n)|}
	= \frac{1}{2}\E \sum_{v: v\sim o_n} \mathbf{1}_{\{(o_nv) \in \M_{G_n}\}}.
\end{align*}

Since $\big (G, o, (w_e)_{e\in E(G)}, (\mathbf{1}_{\{e \in \M_{G}\}})_{e \in E(G)}\big ) \mapsto \sum_{v: v\sim o} w_{(o v)} \mathbf{1}_{\{(o v) \in \M_{G}\}}$ is continuous with respect to the local weak topology on $\mathcal{G}^*_{\mathbb{R}^2}$, by continuous mapping theorem,
\[ \sum_{v: v\sim o_n} w_{(o_nv)} \mathbf{1}_{\{(o_nv) \in \M_{G_n}\}} \stackrel{d}{\to} \sum_{v: v\sim o} w_{(o v)} \mathbf{1}_{\{(o v) \in \M_{G}\}}.\]
Similarly,
\[ \sum_{v: v\sim o_n} \mathbf{1}_{\{(o_nv) \in \M_{G_n}\}} \stackrel{d}{\to} \sum_{v: v\sim o} \mathbf{1}_{\{(o v) \in \M_{G}\}}.\]
To prove Corollary~\ref{thm: lln}, it then suffices to show that 
\begin{equation}
\label{eq: unif_int}
\Big(\sum_{v: v\sim o_n} w_{(o_nv)} \mathbf{1}_{\{(o_nv) \in \M_{G_n}\}}\Big)_{n},  \ \ \Big(\sum_{v: v\sim o_n} \mathbf{1}_{\{(o_nv) \in \M_{G_n}\}}\Big)_{n}   \quad \text{are uniformly integrable.}
\end{equation}
 Note that since $\M_{G_n}$ is a matching, there is at most one term nonzero in the sums above. Therefore, 
\[
\Big(\sum_{v: v\sim o_n} w_{(o_nv)} \mathbf{1}_{\{(o_nv) \in \M_{G_n}\}}\Big)^p = \sum_{v: v\sim o_n} |w_{(o_nv)}|^p \mathbf{1}_{\{(o_nv) \in \M_{G_n}\}} \leq \sum_{v: v\sim o_n} |w_{(o_nv)}|^p.
\]
Hence,
\[
\E \Big(\sum_{v: v\sim o_n} w_{(o_nv)} \mathbf{1}_{\{(o_nv) \in \M_{G_n}\}}\Big)^p \leq \E |w_e|^p \E d_{o_n} =  \E |w_e|^p \E \tfrac{2|E(G_n)|}{|V(G_n)|},
\]
which is bounded above by a constant by the hypothesis. Similarly,
\[
\E \Big(\sum_{v: v\sim o_n}  \mathbf{1}_{\{(o_nv) \in \M_{G_n}\}}\Big)^p \leq   \E \tfrac{2|E(G_n)|}{|V(G_n)|} = O(1).
\]
Therefore, \eqref{eq: unif_int} holds, and this proves Corollary~\ref{thm: lln}.

\appendix

\section{Local weak convergence}
\label{sec: local_weak_convergence}
In this appendix, we recall the definition of local weak convergence. A rooted graph is a pair $(G, o)$, where $G$ is a locally finite graph, and $o\in V(G)$. Write $\mathcal{G}^*$ for the space of isomorphism classes of rooted connected graphs. Define a metric on $\mathcal{G}^*$ as follows:
\[
d_{\mathrm{loc}}((G_1, o_1), (G_2, o_2)) = \frac{1}{R+1},
\]
where
\begin{align*}
	R = \sup\left\{r \ge 0: 
	\begin{array}{l}
		\text{$\exists$ a rooted graph isomorphism $\phi$ between $(\bB_{o_1}^r(G_1), o_1)$ and $(\bB_{o_2}^r(G_2), o_2)$} 
	\end{array}
	\right\}.
\end{align*}
Then $(\mathcal{G}^*, d_{\mathrm{loc}})$ is a Polish space (see \cite[Theorem~A.8]{vdHofstad}). Under this metric, we say that the sequence of finite random rooted graphs $(G_n,o_n)$, where conditional on $G_n$, the root $o_n$ is chosen uniformly from $V(G_n)$, converges locally weakly to a random rooted graph $(G,o)$ with law $\mu$, if for every bounded continuous $h:\mathcal{G}^*\to\R$,
\[
\E h(G_n,o_n)\to \E_\mu h(G,o) \quad \text{ as } n \to \infty.
\]
Let $\mathcal{G}^{**}$ be the space of locally finite doubly rooted graphs up to rooted isomorphism and equip it with its natural local topology. A probability measure $\mu$ on $\mathcal{G}_*$ is unimodular if, for each $f: \mathcal{G}^{**} \to \R$ bounded measurable,
\[
\int \sum_{v\in V(G)} f(G, o, v)~\mathrm{d}\mu(G, o) = \int \sum_{v\in V(G)} f(G, v, o)~\mathrm{d}\mu(G, o).
\]
A graph $(G, o)$ sampled from a unimodular measure is called a unimodular graph. 
A finite random graph with uniformly chosen root is unimodular. It is well-known that the class of unimodular measures is closed under local weak limit. In particular, if the random graph $(G, o)$ arises as a local weak limit, then it is unimodular.

The notion of local weak convergence can naturally be extended to edge-weighted graphs. A weighted rooted graph is a triple $(G, o, (w_e)_{e\in E(G)})$, where $G$ is a locally finite graph, $o\in V(G)$, and for each  $e\in E(G),$ $w_e \in \mathbb{R}^k$ is the edge weight associated with  $e$. Write $\mathcal{G}^*_{\mathbb{R}^k}$ for the space of isomorphism classes of rooted connected weighted graphs. We define the local metric on this space as 
\[
d_{\mathrm{loc}}^w\big((G_1, o_1, (w_e^1)_{e\in E(G_1)}), (G_2, o_2, (w_e^2)_{e\in E(G_2)})\big) = \frac{1}{R+1},
\]
where
\begin{align*}
	R = \sup\left\{r \ge 0: 
	\begin{array}{l}
		\text{$\exists$ a rooted graph isomorphism $\phi$ between $(\bB_{o_1}^r(G_1), o_1)$}\\
		\text{and $(\bB_{o_2}^r(G_2), o_2)$ such that $\sqrt{\sum_{e \in E(\bB_{o_1}^r(G_1))} \|w_e^1 - w_{\phi(e)}^2\|_{2}^2} \leq 2^{-r}$} 
	\end{array}
	\right\}.
\end{align*}
Note that $d_{\mathrm{loc}}^w$ is a metric on $\mathcal{G}^*_{\mathbb{R}^k}$ that makes  $(\mathcal{G}^*_{\mathbb{R}^k},d_{\mathrm{loc}}^w)$ Polish. 
Using this metric, the local weak convergence of a sequence of finite weighted rooted random graphs  $(G_n, o_n, (w_e)_{e\in E(G_n)})$ to $(G, o, (w_e)_{e\in E(G)})$  can be defined similarly.

If $(G_n, o_n)$ converges to  $(G, o)$ locally weakly and 
if the edge weights $(w^{(n)}_e)_{e \in E(G_n)}, (w_e)_{e \in E(G)} $ are drawn i.i.d.\ from a fixed common distribution on $\mathbb{R}^k$, then we have local weak convergence for the corresponding weighted random graphs as well. This is a consequence of the following standard lemma whose proof is omitted. 
\begin{lem} \label{lem: local_weak_convergence}
Suppose $(G_n,o_n)$ converges to $(G,o)$ locally weakly. For each $n$, assign to every edge $e\in E(G_n)$ an independent weight $w_e^{(n)}$ with common law $\nu_n$ on $\R^k$, independently of $G_n$. If $\nu_n \to \nu$ converges weakly, then
\[
(G_n,o_n,(w_e^{(n)})_{e\in E(G_n)})  \text{ converges to } (G,o,(w_e)_{e\in E(G)}) \ \ \text{locally weakly on  } \mathcal{G}^*_{\mathbb{R}^k},
\]
where $(w_e)_{e\in E(G)}$ are i.i.d.\ with law $\nu$, independent of $(G,o)$. In particular, if $\nu_n =  \nu$ for all $n$, the conclusion holds.
\end{lem}
The unimodularity can be analogously defined for weighted random graphs.

\bibliographystyle{plain}
\bibliography{matchref}

\end{document}